\documentclass[a4paper,final]{article}
\usepackage[preprint]{optional}
\usepackage[margin=3.5cm]{geometry}
\usepackage{lmodern}
\usepackage{algorithm,algorithmic,amsmath,amssymb}
\usepackage{graphicx}
\usepackage{ifpdf}
\usepackage{hyperref}

\usepackage{url}
\usepackage{pdfsync}

\usepackage{amsopn}
\usepackage{color}

\usepackage[nocites,norefs]{refcheck}
\usepackage[notref]{showkeys}

\opt{draft}{
\usepackage{todonotes}
}

\newcommand{\toepop}{{\mathcal T}}

\newcommand{\C}{{\mathbb C}}
\newcommand{\R}{{\mathbb R}}
\newcommand{\N}{{\mathbb N}}
\newcommand{\calP}{{\mathcal P}}

\DeclareMathOperator{\diag}{diag}
\DeclareMathOperator{\rank}{rank}
\DeclareMathOperator{\range}{ran}
\DeclareMathOperator{\real}{Re}
\DeclareMathOperator{\imag}{Im}
\DeclareMathOperator{\expmt}{expmt}
\DeclareMathOperator{\expm}{expm}

\newcommand{\norm}[1]{\left\lVert #1 \right\rVert}
\newcommand{\unorm}[1]{\lVert #1 \rVert}
\newcommand{\setunion}{\ensuremath{\,\cup\,}}

\newcommand{\defby}{\mathrel{\mathop:}=}

\newcommand{\bydef}{=\mathrel{\mathop:}}
\DeclareMathOperator{\GL}{GL}

\newcommand{\smb}{\left[\begin{smallmatrix}}
\newcommand{\sme}{\end{smallmatrix}\right]}
\newcommand{\OOM}{\mathcal{O}}

\opt{siam}{
\newsiamremark{example}{Example}
\newsiamremark{remark}{Remark}
\newsiamremark{assumption}{Assumption}
}

\opt{preprint}{
\usepackage{amsthm}
\newtheorem{theorem}{Theorem}[section]
\newtheorem{lemma}[theorem]{Lemma}

\newtheorem{corollary}[theorem]{Corollary}

\newtheorem{definition}[theorem]{Definition}
\newtheorem{example}[theorem]{Example}
\newtheorem{remark}[theorem]{Remark}

}

\newcommand{\titlestrlong}{Fast computation of the matrix exponential for a Toeplitz matrix}
\newcommand{\titlestrshort}{Fast Toeplitz matrix exponential}
\newcommand{\authorstr}{Daniel Kressner and Robert Luce}

\opt{siam}{
\headers{\titlestrshort}{\authorstr}
}

\opt{preprint}{
\newcommand{\email}[1]{{\ttfamily #1}}
}

\date{}

\title{\titlestrlong}

\author{
  Daniel Kressner\thanks{\'Ecole Polytechnique F\'ed\'erale de Lausanne,
    Station 8, 1015 Lausanne, Switzerland
    (\email{daniel.kressner@epfl.ch}, \url{http://anchp.epfl.ch}).}
  \and
  Robert Luce\thanks{\'Ecole Polytechnique F\'ed\'erale de Lausanne,
    Station 8, 1015 Lausanne, Switzerland
    (\email{robert.luce@epfl.ch}, \url{http://people.epfl.ch/robert.luce}).}
}

\ifpdf
\hypersetup{
  pdftitle={\titlestrlong},
  pdfauthor={\authorstr}
}
\fi

\begin{document}

\maketitle

\begin{abstract}
The computation of the matrix exponential is a ubiquitous operation in
numerical mathematics, and for a general, unstructured $n\times n$
matrix it can be computed in $\mathcal{O}(n^3)$ operations.  An
interesting problem arises if the input matrix is a \emph{Toeplitz
matrix}, for example as the result of discretizing integral equations
with a time invariant kernel.  In this case it is not obvious how to
take advantage of the Toeplitz structure, as the exponential of a
Toeplitz matrix is, in general, not a Toeplitz matrix itself. The main
contribution of this work are fast algorithms for the computation of
the Toeplitz matrix exponential.  The algorithms have provable
\emph{quadratic complexity} if the spectrum is real, or sectorial, or
more generally, if the imaginary parts of the rightmost eigenvalues do
not vary too much.  They may be efficient even outside these spectral
constraints.  They are based on the scaling and squaring framework, and
their analysis connects classical results from rational approximation
theory to matrices of low displacement rank.  As an example, the
developed methods are applied to Merton's jump-diffusion model for
option pricing.

\end{abstract}

\opt{siam}{
}

\section{Introduction}

Let us consider an $n\times n$ Toeplitz matrix
\begin{equation} \label{eq:toeplitz}
 T = \begin{bmatrix}
    t_0 & t_{-1} & \cdots & t_{-n+1} \\
    t_1 & t_0 & \ddots & \vdots \\
    \vdots & \ddots & \ddots & t_{-1} \\
    t_{n-1} & \cdots & t_1 & t_0
\end{bmatrix}.
\end{equation}
In this work, we propose a new class of fast algorithms for computing a highly accurate approximation of the matrix exponential $\exp(T)$. 
An important source of applications for $\exp(T)$ arises from the discretization of integro-differential equations with a shift-invariant kernel. Such equations play a central role in, e.g., the pricing of single-asset options modelled by jump-diffusion processes~\cite{Duffy2006,Sachs2008}. 
The related problem of computing the exponential of a block Toeplitz matrix appears in the Erlangian approximation of Markovian fluid queues~\cite{Bini2015}.

It is well known that the multiplication of a Toeplitz matrix with a
vector can be implemented in $\OOM(n\log n)$ operations, using the FFT. This suggests the use of a Krylov subspace method, such as the Lanczos method, for computing the product of $\exp(T)$ with a vector $b$; see, e.g.,~\cite{Ng2000}. For a matrix $T$ of large norm, the Krylov subspace method can be expected to converge slowly~\cite{Hochbruck1997}. In this case, the use of rational Krylov subspace methods is advisable. For example, Lee, Pang, and Sun~\cite{Lee2010} have suggested a shift-and-invert Arnoldi method for approximating $\exp(T)b$. Every step of this method requires the solution of a linear system with a Toeplitz matrix. The fast and superfast solution of such linear systems has received broad attention in the literature; we refer to~\cite{Heinig1984,Kailath1999,Olshevsky2001a,Olshevsky2001b} for overviews. Recent work in this direction includes an algorithm based on a combination of rank structured matrices and randomized sampling~\cite{Xia2012}.

If, additionally, $T$ is upper triangular then $T^2$ and, more generally, any matrix function of $T$ is again an upper triangular Toeplitz matrix. This very desirable property allows for the design of efficient algorithms that directly aim at the computation of generators for $\exp(T)$; see~\cite{Bini2015} and the references therein. It is important to note that this property does not extend to general Toeplitz matrices.

The approach proposed in this work is different from existing
approaches, because it aims at approximating the full matrix
exponential $\exp(T)$, instead of $\exp(T)b$, and it does not impose
additional structure on $T$. Our approach is based on a combination of
the scaling and squaring method for the matrix exponential of
unstructured matrices~\cite{Guttel2016,Higham2008,Higham2009} with
approximations of low displacement rank~\cite{Kailath1999}.
Specifically, we show that the displacement rank of a rational
function of $T$ is bounded by the degree of the rational function. In
turn, we obtain an approximate representation of $\exp(T)$, which
requires $\OOM(n)$ storage under suitable assumptions and allows to
conveniently multiply $\exp(T)$ with a vector in $\OOM(n\log n)$
operations. The latter property is particularly interesting in option
pricing; it allows for quickly evaluating prices for times to maturity
that are integer multiplies of a fixed time period. The availability
of an approximation to the full matrix exponential also allows us to
quickly access parts of that matrix. For example, the diagonal entries
can be computed in $\OOM(n)$ operations, which would be significantly
more expensive using Krylov subspace methods.


\section{Toeplitz matrices} \label{sec:toeplitz}

In this section, we recall and establish basic properties of Toeplitz matrices needed for our developments.
Following~\cite{Kailath1995}, we define the displacement $\nabla_F(A)$ of $A\in \C^{n\times n}$ with respect to $F\in \C^{n\times n}$ as
\[
 \nabla_F(A) := A - F  A F^*.
\]
We will mostly use the downward shift matrix for $F$, in which case we
omit the subscript:
\[
 \nabla(A) = A - Z  A Z^*,
 \qquad Z =  \begin{bmatrix}
    0 &  \\
    1 & 0 & \\
    & \ddots & \ddots  \\
    & & 1 & 0
\end{bmatrix}.
\]
The rank of $\nabla(A)$ is called the \emph{displacement rank} of $A$. Toeplitz matrices have displacement rank at most two.
Matrices of ``small'' displacement rank are often called \emph{Toeplitz-like} matrices.
Given an invertible matrix $A$ with $\rank(\nabla(A)) \le r$, it follows that
$\rank(\nabla(A^{-1})) \le r + 2$. More generally, any Schur complement of
$A$ has bounded displacement rank, see~\cite[Thm.~2.2]{Kailath1995}.
These displacement rank properties are discussed with great detail in
Section~\ref{sec:toeplk_drank}.

It follows that the inverse of a Toeplitz matrix $T$ has displacement
rank at most $2$. This property does \emph{not} extend
to general matrix functions of $T$.  In particular, $\exp(T)$ usually
has full displacement rank. However, as we will see below in Section~\ref{sec:approxexponential}, it turns
out that matrix functions of $T$ can often be well approximated by a matrix of low displacement rank.

\subsection{Generators, reconstruction, and fast matrix-vector products}

For $r \ge \rank(\nabla(A))$, there are matrices $G,B \in
\C^{n\times r}$ such that
\begin{equation} \label{eq:stein}
  \nabla(A) = A - Z A Z^* = G B^*.
\end{equation}
We call such a pair $(G,B)$ a \emph{generator} for $A$. Note that
$A$ admits many generators and $(G,B)$ is called a
\emph{minimal} generator for $A$ if $r = \rank(\nabla(A))$. 
A generator for the Toeplitz matrix~\eqref{eq:toeplitz} is given by
\begin{equation} \label{eq:toeplitzgenerator}
  G = \begin{bmatrix}
      t_0 & 1 \\
      t_1 & 0 \\
      \vdots & \vdots \\
      t_{n-1} & 0
     \end{bmatrix}, \qquad 
B = \begin{bmatrix}
     1 & 0 \\
     0 & \bar t_{-1} \\
     \vdots & \vdots \\
     0 & \bar t_{-n+1}
    \end{bmatrix}.
\end{equation}

Fast algorithms for Toeplitz-like matrices operate
directly on the generator of $A$ instead of $A$ itself. When needed,
the full matrix can be reconstructed from the generators by
noting that~\eqref{eq:stein} is a matrix Stein equation admitting the
unique solution
\begin{equation} \label{eq:reconstruction}
  A = \toepop(G,B) := \sum_{k = 0}^{n-1} Z^k G B^* (Z^*)^k.
\end{equation}
Letting $g_j, b_j \in \C^n$ for $j = 1,\ldots, r$ denote the columns of $G,B$,
we can rewrite~\eqref{eq:reconstruction} as
\begin{equation}
\label{eq:disptoeplitz}
    A = \toepop(G,B) = L(g_1) U(b_1^*) + L(g_2) U(b_2^*) + \cdots  + L(g_r) U(b_r^*),
\end{equation}
with the triangular Toeplitz matrices
\begin{equation*}
 L(x):= \begin{bmatrix}
    x_1 & 0 & \cdots & 0 \\
    x_2 & x_1 & \ddots & \vdots \\
    \vdots & \ddots & \ddots & 0 \\
    x_{n} & \cdots & x_2 & x_1
\end{bmatrix}, \qquad 
U(x):=\begin{bmatrix}
    x_1 & x_{2} & \cdots & x_{n} \\
    0 & x_1 & \ddots & \vdots \\
    \vdots & \ddots & \ddots & x_{2} \\
    0 & \cdots & 0 & x_1
\end{bmatrix}.
\end{equation*}

Using the Fast Fourier Transform (FFT), the matrix-vector product with
a Toep\-litz matrix can be done in
$O(n\log n)$ operations; see, e.g.,~\cite[Sec. 10.2.3]{Bai2000}.
In turn,~\eqref{eq:disptoeplitz} shows that the
multiplication of $A$ with a vector can be computed in $O(r n\log n)$
operations; see~\cite[Chap.~1]{Kailath1999} for more details.

\subsection{Generator Truncation}

Operations like matrix addition or multiplication typically increase
the displacement rank of Toeplitz-like matrices.  Even worse, the
result of such operations may lead to non-minimal, that is, rank
deficient generators with many more columns than necessary. To
limit this increase in generator size, we will truncate the singular
values of the generators. The following results justify our
procedure.


\begin{lemma}
\label{lem:norm_bound}
The displacement $\nabla(A)$ for $A \in \C^{n\times n}$ satisfies
\begin{equation*}
    \frac{1}{2} \norm{\nabla(A)}_* \le \norm{A}_* \le n \norm{\nabla(A)}_*,
\end{equation*}
where $\norm{\cdot}_*$ denotes any unitarily invariant norm.
\end{lemma}
\begin{proof}
Note that $\norm{Z^{k}}_2 = 1$ for $0 \le k<n$.  The first inequality
follows directly from the definition of the displacement operator,
viz.
\begin{equation*}
    \norm{\nabla(A)}_*
    = \norm{A - Z A Z^*}_*
    \le \norm{A}_* + \norm{Z}_2 \norm{A}_* \norm{Z}_2
    \le 2 \norm{A}_*.
\end{equation*}
For the second inequality we compute
from~\eqref{eq:reconstruction} that
\begin{equation*}
    \norm{A}_* \le \sum_{k=0}^{n-1} \norm{Z^k \nabla(A) (Z^*)^k}_2
    \le \sum_{k=0}^{n-1} \norm{Z^k}_2 \norm{\nabla(A)}_* \norm{(Z^*)^k}_2
    = n \norm{\nabla(A)}_*.
\end{equation*}{}
\end{proof}

The bounds of Lemma~\ref{lem:norm_bound} may not be sharp. In particular, one may question whether the factor $n$ of the upper bound is necessary. The following example shows that this linear dependence on $n$ can, in general, not be removed.

\begin{example}
Let $g = [1,1,\dotsc,1]^* \in \R^n$ and $A = \toepop(gg^*)$. Then the $k$th entry of $f = Ag$ is given by
$f(k) = kn - k(k-1)/2$. Since $f$ is monotonically increasing, this allows us to estimate
\[
 \|f\|^2_2 = \sum_{k = 1}^n f(k)^2 \ge \int_0^n f(k)^2\,\mathrm{d}k \ge \frac{2}{15} n^5.
\]
In turn,
\[
 \|A\|_2 \ge \frac{\|A g\|_2}{\|g\|_2} \ge \sqrt{\frac{2}{15}} n^2 = \sqrt{\frac{2}{15}} n \|\nabla(A)\|_2,
\]
which shows that $\|A\|_2 / \|\nabla(A)\|_2$ grows linearly with $n$.
\end{example}

Lemma~\ref{lem:norm_bound} allows us to analyze the effect of
generator truncation in terms of the approximation error.
\begin{theorem} \label{thm:compresstoeplitz}
Let $A \in \C^{n \times n}$ be a Toeplitz-like matrix of displacement
rank $r$ and consider the singular value decomposition (SVD)
\begin{equation*}
    \nabla(A) = U \Sigma V^*
    = \begin{bmatrix}
        U_1 & U_2
    \end{bmatrix}
    \begin{bmatrix}
        \Sigma_1 & 0\\
        0        & \Sigma_2
    \end{bmatrix}
    \begin{bmatrix}
        V_1^*\\
        V_2^*
    \end{bmatrix},
\end{equation*}
where $\Sigma_1 \in \diag(\sigma_{1}, \dotsc,
\sigma_{\tilde{r}})$ and $\Sigma_2 = \diag(\sigma_{\tilde{r}+1}, \dotsc,
\sigma_r)$.  Letting $\tilde A  = \toepop(U_1 \Sigma_1,V_1)$, it holds that
\begin{equation} \label{eq:truncbounds}
    \| A - \tilde A\|_2 \le n \sigma_{\tilde{r}+1}
    \quad \text{and} \quad
    \|A-\tilde A\|_F
    \le n \sqrt{\sigma_{\tilde{r}+1}^2 + \dotsb + \sigma_r^2}.
\end{equation}
\end{theorem}
\begin{proof}
By linearity of the displacement operator $\nabla$ we find
\begin{equation*}
    \nabla(A - \tilde A)
    = \nabla(A) - \nabla(\tilde A)
    = U \Sigma V^* - U_1 \Sigma_1 V_1^*
    = U_2 \Sigma_2 V_2^*.
\end{equation*}
Hence, the claimed bounds follow from applying Lemma~\ref{lem:norm_bound} to $A - \tilde A$.
\end{proof}

\noindent Note that~\eqref{eq:truncbounds} improves upon a result by Pan~\cite[Eq. (3.5)]{Pan1993}.

We will use the construction of Theorem~\ref{thm:compresstoeplitz} to compress a generator $(G,B)$ with $G,B \in \C^{n\times r}$ to a generator  $(\tilde G,\tilde B)$ with $\tilde G,\tilde B \in \C^{n\times \tilde r}$ and $\tilde r < r$. By~\eqref{eq:truncbounds}, the singular values of $GB^*$ allow us to quantify the compression error and choose $\tilde r$ adaptively. In the particular case of a non-minimal (rank deficient) generator $(G,B)$ with $r > \tilde r = \rank(GB^*)$, the construction returns an exact minimal generator.
Typically we have $r\ll n$, in which case the cost greatly reduces from $\OOM(n^3)$ FLOPs for computing the SVD of $GB^*$ to $\OOM(r^2 n + r^3)$ FLOPs by first computing thin QR decompositions of $G$ and $B$. This well-known procedure is summarized in Algorithm~\ref{alg:gencompress}.

\begin{algorithm}
\caption{Generator Compression}
\label{alg:gencompress}
\begin{algorithmic}[1]
\REQUIRE Generator matrices $G,B \in \C^{n \times r}$, integer $\tilde r < r$.
\ENSURE Generator matrices $\tilde{G}, \tilde{B} \in \C^{n \times \tilde r}$ such
that $\toepop(G,B) \approx \toepop(\tilde G,\tilde B)$.
\STATE{Compute thin QR factorizations $Q_G R_G = G$ and $Q_B R_B = B$}
\STATE{Compute $S = R_G R_B^* \in \C^{r\times r}$}
\STATE{Compute truncated SVD $U_1 \Sigma_1 Y_1^* \approx S$ with $\Sigma_1 \in \R^{\tilde r \times \tilde r}$}
\STATE{Set $\tilde{G} = Q_GX_1 \Sigma_1^{\frac{1}{2}}$ and
$\tilde{B} = Q_B Y_1 \Sigma_1^{\frac{1}{2}}$}
\end{algorithmic} \end{algorithm}

Generators are not uniquely determined. For every $Z \in \GL_r(\C)$, the generators $(G,B)$ and $(GZ, BZ^{-*})$ correspond to the same Toeplitz-like matrix. The following lemma shows that this relation in fact characterizes the set of all minimal generators.
\begin{lemma}
Let $A \in \C^{n \times n}$ be a Toeplitz-like matrix of displacement
rank $r$, and let $(G_1, B_1)$, $(G_2, B_2)$ be two minimal generators
for $A$.  Then there exists a matrix $Z \in \GL_r(\C)$ such that $G_1 =
G_2 Z$ and $B_1 = B_2 Z^{-*}$. 
\end{lemma}
\begin{proof}
    By minimality of the generators all of $G_i$ and $B_i$, $i=1,2$,
    have full rank.  Hence $G_1$ is right-equivalent to $G_2$, and
    $B_1$ is right-equivalent to $B_2$, i.e., there exist $Z,W \in
    \GL_r(\C)$ such that $G_1 = G_2 Z$ and $B_1 = B_2 W$.  But then
    \begin{equation*}
        G_2 B_2^* = G_1 B_1^* = G_2 Z W^* B_2^*,
    \end{equation*}
    so $G_2 (I - Z W^*) B_2^* = 0$.  Since $G_2$ and $B_2$ have full
    rank it follows that $W = Z^{-*}$.
\end{proof}

%
%
%
%
%
%

\section{Bounds on the displacement rank of functions of Toeplitz matrices} \label{sec:rationalapproximation}

The scaling and squaring
method~\cite[chap.~10]{Higham2008} for the
evaluation of $\exp(T)$ takes three phases:  First, the matrix $T$ is
scaled by a power of two, then a Pad\'e approximant of the scaled matrix is
computed, and in a third step, the approximant is repeatedly squared
in order to undo the initial scaling.
In the context of Toeplitz matrices, the main challenge is to control the growth of the displacement rank in the second and third phase.


\subsection{Polynomial and rational functions of Toeplitz matrices}
\label{sec:toeplk_drank}

In the following, we analyze the impact of various operations on the displacement rank of Toeplitz-like matrices and provide explicit expressions for the resulting generators. The following result is a variation of the well-known result that Schur complements do not increase the displacement rank.

\begin{lemma}[Generator block update]
\label{lem:gen_block_update}
Consider $M = \smb D & U \\ L & M_1 \sme \in \C^{n \times n}$ with 
$D \in \C^{k\times k}$ invertible and 
\[
 \nabla_F(M) = G B^*, \qquad G,B \in \C^{n\times r},
\]
for a strictly lower triangular matrix $F$. Let us partition
$
 F = \big[ \begin{smallmatrix}
 \hat{F} & 0 \\
      \star & F_1
     \end{smallmatrix} \big]
$ with $\hat{F} \in \C^{k\times k}$, $F_1 \in \C^{(n-k) \times
(n-k)}$, and
$G = \big[ \begin{smallmatrix} \hat{G} \\ \star \end{smallmatrix} \big]$, 
$B = \big[ \begin{smallmatrix} \hat{B} \\ \star \end{smallmatrix} \big]$
    with $\hat{G}, \hat{B} \in \C^{k \times r}$ ($\star$ is used as a placeholder referring to an arbitrary block).
Then the Schur complement of $D$ in $M$ satisfies
\[
 \nabla_{F_1}(M_1 - L D^{-1} U) = G_1 B_1^*,
\]
where the generator matrices $G_1, B_1 \in \R^{(n-k) \times r}$ are defined by the relations
\begin{align}
\label{eq:G_update}
    \begin{bmatrix}
        0\\
        G_1
    \end{bmatrix}
    & =
    G + (F - I_n)
    \begin{bmatrix}
        D\\
        L
    \end{bmatrix}
    D^{-1} (I_k - \hat{F})^{-1} \hat{G},\\
\label{eq:B_update}
    \begin{bmatrix}
        0\\
        B_1
    \end{bmatrix}
    & =
    B + (F - I_n)
    \begin{bmatrix}
        D^*\\
        U^*
    \end{bmatrix}
    D^{-*} (I_k - \hat{F})^{-1} \hat{B}.
\end{align}
\end{lemma}
\begin{proof}
The result is a direct extension of~\cite[Alg.~3.3]{Sayed1995} from the Hermitian to the non-Hermitian case.
\end{proof}

It is well known that the displacement rank of the product $T_1 T_2$ of two Toeplitz matrices $T_1,T_2$ is at most $4$~\cite[Example 2]{Kailath1994}.
The following theorem extends this result to Toeplitz-like matrices.
\begin{theorem} 
\label{thm:product_generators}
Let $A_1,A_2 \in \C^{n,n}$ be two Toeplitz-like matrices of displacement ranks $r_1,
r_2$ with
generators $(G_1, B_1)$ and $(G_2, B_2)$, respectively.  Then $A_1 A_2$ is a Toeplitz-like matrix of displacement rank
at most $r_1 + r_2 + 1$, and a generator $(G,B)$ for $A_1 A_2$ is
given by
\begin{align*}
    G & =
    \begin{bmatrix}
        (Z-I) A_1 (Z-I)^{-1} G_2 & G_1 & -(Z-I) A_1 (Z-I)^{-1} e_1
    \end{bmatrix},\\
    \quad
    B & =
    \begin{bmatrix}
        B _2 & (Z-I) A_2^* (Z-I)^{-1} B_1 & (Z-I) A_2^* (Z-I)^{-1} e_1
    \end{bmatrix},
\end{align*}
where $e_1 \in \R^n$ denotes the first unit vector. If, additonally, $e_1 \in \range(G_2) \setunion \range(B_1)$ then $A_1 A_2$ has displacement rank at most $r_1+ r_2$.
\end{theorem}
\begin{proof}
Consider the matrix
\begin{equation*}
    M =
    \begin{bmatrix}
        -I  & A_2\\
        A_1 & 0
    \end{bmatrix},
\end{equation*}
and set $F = Z \oplus Z$.  One computes that
\begin{equation*}
    \nabla_F(M) = M - F M F^* =
    \begin{bmatrix}
        -e_1 e_1^* & G_2 B_2^*\\
        G_1 B_1^*  & 0
    \end{bmatrix}
    =
    \begin{bmatrix}
        G_2 & 0   & -e_1\\
        0   & G_1 & 0
    \end{bmatrix}
    \begin{bmatrix}
        0     & B_2^*\\
        B_1^* & 0\\
        e_1^* & 0
    \end{bmatrix}.
\end{equation*}
Since
$A_1 A_2$ is the Schur complement of $-I$ in $M$, Lemma~\ref{lem:gen_block_update} implies that $A_1 A_2$ has displacement rank at most $r_1+r_2+1$.
Moreover, by~\eqref{eq:G_update}--\eqref{eq:B_update}, a generator $(G, B)$ for $A_1 A_2$ is given by 
\begin{align*}
    G & =
    \begin{bmatrix}
        0 & G_1 & 0
    \end{bmatrix}
    + (Z - I) A_1 (Z - I)^{1}
    \begin{bmatrix}
        G_2 & 0 & -e_1
    \end{bmatrix}\\
    & =
    \begin{bmatrix}
        (Z-I) A_1 (Z-I)^{-1} G_2 & G_1 & -(Z-I) A_1 (Z-I)^{-1} e_1
    \end{bmatrix},\\
    B & =
    \begin{bmatrix}
        B_2 & 0 & 0
    \end{bmatrix}
    + (Z - I) A_2^* (Z - I)^{-1}
    \begin{bmatrix}
        0 & B_1 & e_1
    \end{bmatrix}\\
    & =
    \begin{bmatrix}
        B _2 & (Z-I) A_2^* (Z-I)^{-1} B_1 & (Z-I) A_2^* (Z-I)^{-1} e_1
    \end{bmatrix}.
\end{align*}
Note that at least one of these matrices becomes rank deficient if $e_1 \in \range(G_2) \setunion \range(B_1)$, which shows the second part of the theorem.
\end{proof}
Because of~\eqref{eq:toeplitzgenerator}, the additional condition of Theorem~\ref{thm:product_generators} is satisfied for Toeplitz matrices. An analogous condition plays a role in controlling the displacement rank for the inverse of a Toeplitz-like matrix.
\begin{theorem}
\label{thm:inverse_tl}
Let $A$ be an invertible Toeplitz-like matrix of displacement rank $r$ with
generator $(G,B)$. Then $A^{-1}$ is a Toeplitz-like matrix of
displacement rank at most $r+2$, and a generator $(\tilde{G},
\tilde{B})$ is given through
\begin{align*}
    \tilde{G} & = 
    \begin{bmatrix}
        -(Z-I) A^{-1} (Z-I)^{-1} G & 
         (Z-I) A^{-1} (Z-I)^{-1} e_1 & 
        e_1
    \end{bmatrix},\\
    \tilde{B} & = 
    \begin{bmatrix}
        (Z-I) A^{-*} (Z-I) B &
        e_1 &
        (Z-I) A^{-*} (Z-I) e_1
    \end{bmatrix}.
\end{align*}
If, additionally, $e_1 \in \range(G) \cup \range(B)$ then $A^{-1}$ has displacement rank at most $r+1$.
\end{theorem}
\begin{proof}
The result follows from applying the technique from the proof of 
Theorem~\ref{thm:product_generators} to the embedding $M = \smb -A &
I\\ I & 0 \sme$, which has the generator $\hat G = \smb G & 0 & e_1 \\ 0 & e_1 & 0 \sme$, $\hat B = \smb B & e_1 & 0 \\ 0 & 0 & e_1 \sme$.
The second part follows from the observation that at least one of
$\hat G, \hat B$
has at most rank $r+1$ if $e_1 \in \range(G) \cup \range(B)$.
\end{proof}

\begin{remark} \label{remark:toeplitzinverse}
 We remark that the special shape~\eqref{eq:toeplitzgenerator} of $B,G$ for a Toeplitz matrix $T$ imply that the matrix $\hat G \hat B^*$ in the proof of Theorem~\ref{thm:inverse_tl} has rank $\le 2$ and, in turn, the displacement rank of $T^{-1}$ is $\le 2$.
 In fact, letting $(G,B) = \left( \smb g & e_1 \sme, \smb e_1 & b\sme \right)$ denote the generator~\eqref{eq:toeplitzgenerator} for $T$, the matrices
 \begin{align*}
    \hat{G} & = \begin{bmatrix} e_1 & 0 \end{bmatrix}
        + (Z-I) T^{-1} (Z-I)^{-1} \begin{bmatrix} -g & e_1 \end{bmatrix},\\
    \quad
    \hat{B} & = \begin{bmatrix} 0 & e_1 \end{bmatrix}
        + (Z-I) T^{-*} (Z-I)^{-1} \begin{bmatrix} e_1&  -b \end{bmatrix}
\end{align*}
constitute a generator for  $T^{-1}$. This result is well known and a corollary of Lemma~\ref{lem:gen_block_update}.
\end{remark}

While Theorems~\ref{thm:product_generators} and~\ref{thm:inverse_tl} are variations of well-known results, we are not aware of existing results on the displacement ranks of powers and polynomials of Toeplitz matrices analyzed in the following.
\begin{lemma}
\label{lem:power_generators}
Let $T$ be a Toeplitz matrix.
Then $T^s$ is a Toeplitz-like matrix of displacement rank at
most $2s$ for any integer $s \ge 1$. Letting $(G, B)$ denote a generator for $T$, a
sequence of (non-minimal) generators $(G_1,B_1), \dotsc, (G_s,B_s)$ for $T, T^2,
\dotsc, T^s$ is given by
\begin{align}
    \label{eq:power_update_G}
    G_1 = G,\quad G_{i+1} & =
    \begin{bmatrix}
        P_G^i G & P_G^{i-1} G & \hdots & G & -P_G e_1 & \hdots & -P_G^i e_1
    \end{bmatrix}\\
    \label{eq:power_update_B}
    B_1 = B, \quad B_{i+1} & =
    \begin{bmatrix}
        B & P_B B & \hdots & P_B^i B & P_B^i e_1 & \hdots & P_B e_1
    \end{bmatrix},
\end{align}
for $i = 1,\ldots,s-1$, where $P_G \defby (Z-I)T(Z-I)^{-1}$ and $P_B \defby
(Z-I)T^*(Z-I)^{-1}$. Moreover, 
\begin{equation}
    \label{eq:generator_nesting}
    e_1 \in \range(G_1) \subset \dotsb \subset \range(G_s)
    \quad \text{and} \quad
    e_1 \in \range(B_1) \subset \dotsb \subset \range(B_s).
\end{equation}
\end{lemma}
\begin{proof}
The proof is by induction on $i$. For $G_1,B_1$, the claim~\eqref{eq:generator_nesting} follows directly from the  expression~\eqref{eq:toeplitzgenerator} for the generator of $T$.

Now assume that~\eqref{eq:power_update_G}--\eqref{eq:generator_nesting} hold for
all $G_i,B_i$ with $i < s$.  Invoking Theorem~\ref{thm:product_generators} with
$T_1 = T^i$ and $T_2 = T$ yields the
formulas~\eqref{eq:power_update_G}--\eqref{eq:power_update_B}.
Further, since all columns of $G_i$ and $B_i$ are also
columns of $G_{i+1}$  and $B_{i+1}$, respectively, we directly obtain $\range(G_i)
\subset \range(G_{i+1})$  and $\range(B_i) \subset \range(B_{i+1})$.

Finally, since $e_1 \in \range(G)$, each of the last $i$ columns of
$G_{i+1}$ and $B_{i+1}$ is a linear combination of one of the first $i+1$
block columns. In turn, 
\begin{align*}
    \range(G_{i+1}) &= \range\left(
    \begin{bmatrix}
        P_G^i G & P_G^{i-1} G & \hdots & G
    \end{bmatrix} \right), \\
    \range(B_{i+1}) &= \range\left(
    \begin{bmatrix}
        B & P_B B & \hdots & P_B^i B
    \end{bmatrix} \right),
\end{align*}
which implies $\rank(G_{i+1}) \le 2(i+1)$ and $\rank(B_{i+1}) \le
2(i+1)$. In particular, the displacement rank of $T^{i+1}$ is at most $2(i+1)$.
\end{proof}

\begin{theorem}
\label{thm:drank_p(T)}
Let $T$ be a Toeplitz matrix and $p \in \calP_s$, where $\calP_s$ denotes the set of polynomials of degree at most $s$. Then $p(T)$ is a Toeplitz-like matrix of displacement
rank at most $2s$.
\end{theorem}
\begin{proof}
Let $p = \sum_{i=0}^s a_k z^k$ and consider the generators
$(G_i,B_i)$, $1 \le i \le s$, for the monomials $T^i$ constructed in~\eqref{eq:power_update_G}--\eqref{eq:power_update_B}. Setting $G_0 := B_0 := e_1$ and using the linearity of the displacement operator $\nabla$ we obtain
\begin{equation*}
    \nabla(p(T))
    = \sum_{i=0}^s a_i \nabla(T^i)
    = \sum_{i=0}^s a_i G_i B_i^*
    =
    \begin{bmatrix}
        a_0 e_1 & a_1 G_1 & \dotsc & a_s G_s
    \end{bmatrix}
    \begin{bmatrix}
        e_1^*\\
        B_1^*\\
        \vdots\\
        B_s^*
    \end{bmatrix}
    \bydef G_p B_p^*.
\end{equation*}
It follows from~\eqref{eq:generator_nesting} that $\rank(G_p) = \rank(G_s) \le
2s$, $\rank(B_p) = \rank(B_s) \le 2s$, and hence $\rank(\nabla(p(T))) \le 2s$.
\end{proof}


The following theorem is the main result of this section and quantifies the effect of a rational function on the displacement rank. It shows that the displacement rank grows at most linearly with the degree of the rational function, defined as the maximal degree of the numerator and denominator.
\begin{theorem}
\label{thm:drank_r(T)}
Let $T$ be a Toeplitz matrix, and let $p \in \calP_{s_p}$, $q \in \calP_{s_q}$ be such that
$q(T)$ is invertible.  Let $(G_p, B_p)$ and $(G_q, B_q)$
denote generators of $p(T)$ and $q(T)$,
respectively.  Then $r(T) = \frac{p(T)}{q(T)}$ is a Toeplitz-like
matrix of displacement rank at most $2 \max\{s_p, s_q\}+1$, and a generator is
given by
\begin{align}
\label{eq:rat_gen_G}
G & =
\begin{bmatrix}
-(Z-I) q(T)^{-1} (Z-I)^{-1} G_q & (Z-I) q(T)^{-1} (Z-I)^{-1} G_p &e_1,
\end{bmatrix}\\
\label{eq:rat_gen_B}
B & =
\begin{bmatrix}
(Z-I) p(T)^* q(T)^{-*} (Z-I)^{-1} B_q &
    B_p &
    (Z-I) p(T)^* q(T)^{-*} (Z-I)^{-1} e_1
\end{bmatrix}.
\end{align}
\end{theorem}
\begin{proof}
The Schur complement of the leading diagonal block in the embedding $M
= \smb -q(T) & p(T)\\ I & 0 \sme$, is $q(T)^{-1} p(T)$, and setting $F
= Z \oplus Z$ one computes
\begin{equation} \label{eq:generatorextended}
M - F M F^* =
\begin{bmatrix}
    -G_q B_q^* & G_p B_p^*\\
    e_1 e_1^* & 0
\end{bmatrix}
=
\begin{bmatrix}
    -G_q & G_p & 0\\
    0 & 0 & e_1
\end{bmatrix}
\begin{bmatrix}
    B_q^* & 0\\
    0 & B_p^*\\
    e_1^* & 0
\end{bmatrix}.
\end{equation}
The formulas~\eqref{eq:rat_gen_G} and~\eqref{eq:rat_gen_B} are
obtained by applying Lemma~\ref{lem:gen_block_update}. 

To see that the matrix~\eqref{eq:generatorextended} has rank at most $2 \max\{s_p, s_q\}+1$, we recall from~\eqref{eq:generator_nesting} that the ranges of the
generator matrices for monomials are nested and thus Theorem~\ref{thm:drank_p(T)} implies $\rank \smb -G_q & G_p \sme \le 2 \max\{s_p, s_q\}$.
\end{proof}

\subsection{Low displacement rank approximation of matrix exponential} \label{sec:approxexponential}

If the singular values of a matrix are rapidly decaying, the limits of
finite precision arithmetic effect that the \emph{numerical rank} of
the matrix is smaller than its rank.  In order to formulate
quantitative statements involving the numerical rank, we use the
following notion of $\varepsilon$-rank.
\begin{definition}
    Let $A \in \C^{n,n}$ and $\varepsilon > 0$.  We say that $A$ has
    $\varepsilon$-rank $k$, if
    \begin{equation*}
        \min_{B \in \C^{n,n}} \{ \norm{A - B}_2 : \rank(B) \le k \}
        \le \varepsilon,
    \end{equation*}
    or, equivalently, if the $k+1$st singular value of $A$ does not
    exceed $\varepsilon$.  The matrix $A$ is said to have
    $\varepsilon$-displacement rank $k$, if $\nabla(A)$ has
    $\varepsilon$-rank $k$.
\end{definition}

Theorem~\ref{thm:drank_p(T)} allows us to derive a priori bounds on the numerical displacement rank of $\exp(T)$, using rational approximations of the exponential function. To see this, let us first recall a seminal result by Gonchar and Rakhmanov~\cite{Gonchar1987}.
\begin{theorem} \label{thm:gonchar} There is a constant $C$ such that
\[
 \inf_{p_1,p_2 \in \calP_s} \max_{\lambda \in (-\infty,0]} | e^\lambda - p_1(\lambda) / p_2(\lambda) | \le C\, V^{-s}
\]
holds for all $s\ge 1$ with $V\approx 9.28903\ldots$.
\end{theorem}

\begin{corollary} \label{cor:approxhnd}
Let $T\in \C^{n,n}$ be a diagonalizable Toeplitz matrix with all eigenvalues real and contained in $(-\infty,\mu]$ for some $\mu \in \R$. Then
\[
  \min\{\norm{\exp(T)-A}_2 :\, \rank(\nabla(A)) \le 2s+1 \} \le \tilde C\,  V^{-s},
\]
for $\tilde C = \kappa(X) e^\mu C$, where 
$C,V$ are as in Theorem~\ref{thm:gonchar} and $\kappa(X)$ is the condition number of a matrix $X$ such that $X^{-1} T X$ is diagonal.
\end{corollary}
\begin{proof}
 According to Theorem~\ref{thm:drank_r(T)}, the matrix $A = e^\mu p_2(T)^{-1} p_1(T)$ with $p_1,p_2 \in \calP_s$ has displacement rank at most $2s+1$.
From
\begin{equation*}
\begin{split}
\norm{\exp(T)-A}_2 
    & = \norm{\exp(\mu I)\exp(T-\mu I ) - A}_2\\
    & \le \kappa(X) e^\mu \max_{\lambda \in
        (-\infty,0]} | e^\lambda - p_1(\lambda) / p_2(\lambda) |,
\end{split}
\end{equation*}
 the result follows using Theorem~\ref{thm:gonchar}.
\end{proof}

Corollary~\ref{cor:approxhnd} implies that the singular values of
$\nabla(\exp(T))$ decay at least exponentially to zero, with a decay
rate that does not deteriorate even if $T$ has very small eigenvalues.
This property is retained by approximations to the matrix exponential.
\begin{corollary}
\label{cor:approx_exp}
Under the assumptions of Corollary~\ref{cor:approxhnd}, let $B \in \C^{n\times n}$ satisfy
$\norm{B - \exp(T)}_2 \le \tau$ for $\tau \ge 0$. If $s$ is an integer such that
\begin{equation*}
    \tilde C V^{-s} \le \tau,
\end{equation*}
then $B$ has $2\tau$-displacement rank $2s+1$.
\end{corollary}
\begin{proof}
The result follows from the triangular inequality
\begin{equation*}
    \norm{B - A}_2
    \le \norm{B-\exp(T)}_2 + \norm{\exp(T) - A}_2 \le 2\tau,
\end{equation*}
and Corollary~\ref{cor:approxhnd}.
\end{proof}

The case of complex spectra is more
difficult. A common approach to obtain rational approximations is
to consider the contour integral representation
\begin{equation} \label{eq:contour}
  \exp(T) = \frac{1}{2\pi \mathrm{i}} \int_\Gamma  e^z (zI - T)^{-1} \,\mathrm{d}z,
\end{equation}
where $\Gamma$ is a contour enclosing the spectrum of $T$.
Applying numerical quadrature with $s$ points to~\eqref{eq:contour} yields an approximation $r(T) \approx \exp(T)$, where $r$ is a rational function of degree $s$ and hence $r(T)$ has displacement rank at most $2s+1$ by Theorem~\ref{thm:drank_r(T)}. In the absence of information on the spectrum of $T$, one might choose $\Gamma$ to be a circle of radius larger than $\|T\|_2$. Applying the composite trapezoidal rule yields exponential convergence but the convergence rate deteriorates as $\|T\|_2$ grows; see, e.g.,~\cite{Trefethen2014}. Sometimes, much better results can be obtained if more information on the spectrum is available. For example, if $A$ is sectorial (that is, its eigenvalues are contained in a sector strictly contained in the left half plane), L\'opez-Fern\'andez et al.~\cite[Thm. 1]{LopezFernandez2006} establish a bound of the form 
\begin{equation*}
\|\exp(T) - r(T) \|_2 \le C\, \gamma^s
\end{equation*}
where the rate $0<\gamma < 1$ depends on the opening angle of the sector but not on the norm of $A$. The rational function $r$ has degree $s$ and is obtained by applying quadrature to~\eqref{eq:contour} with $\Gamma$ chosen to be the left branch of a hyperbola. Analogous results hold for the case that the numerical range of $A$ is contained in the open left half complex plane; see~\cite[Sec. 4.2]{Guttel2013a} for an overview. 

If $T$ has large norm and eigenvalues on or close to the imaginary axis then it cannot be expected that 
$\exp(T)$ admits a good approximation of low displacement rank. In
turn, the methods developed in this paper are not efficient, i.e.,
of quadratic complexity in $n$, for this type of matrices.
The following example illustrates such a situation. 

\begin{figure}
    \begin{center}
        \includegraphics[width=.5\textwidth]{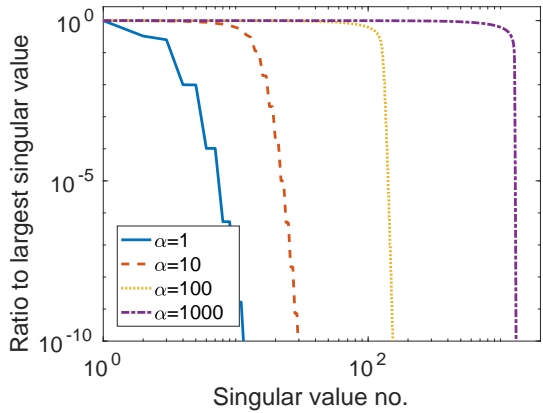}
    \end{center}
    \caption{Decay of the singular values of $\nabla\big( \exp(\alpha T) \big)$ for
    $\alpha \in \{1, 10, 100, 1000\}$ and the Toeplitz matrix $T$
    from Example~\ref{ex:oscillation}.  The plots show the singular value ratios
    $\sigma_j/\sigma_1$ vs. $1 \le j \le 2000$ for each choice of $\alpha$.\label{fig:oscillation}}
\end{figure}

\begin{example} 
 Let $T \in \R^{2000\times 2000}$ be a skew-symmetric Toeplitz matrix~\eqref{eq:toeplitz} with $t_{1} = 1$, $t_{-1} = -1$ and all other entries zero.
    The following table shows the numerical displacement rank of
    $\exp(\alpha T)$, which we compute as the number of singular values of $\nabla(\exp(\alpha T))$ larger than $10^{-10}$ times the first singular value (see also Figure~\ref{fig:oscillation}):
 \[
 \begin{array}{ccccc} \hline 
  \alpha & 1 & 10 & 100 & 1\,000 \\
  \text{num. displacement rank} & 11 &  29 &  153 &    1309 \\ \hline
 \end{array}
 \]
 Clearly, as $\alpha$ grows it becomes increasingly difficult to approximate $\exp(\alpha T)$ by a Toeplitz-like matrix.
\label{ex:oscillation}
\end{example}

\section{Algorithmic tools}
\label{sec:algtools}

In Section~\ref{sec:s_and_s}, we will adapt two variants of the scaling and
squaring method for Toeplitz matrices.  The algorithmic tools
needed for an efficient implementation are the same for both,
and we will describe them in this section without making reference to either algorithm.

\subsection{Norm estimation and scaling}
\label{sec:toep_norm}

The first step of the scaling and squaring method consists of determining a scaling parameter $\rho \in \N$ such that
$\norm{2^{-\rho} T} \approx 1$, which necessitates computing
$\norm{T}$ or an estimate thereof.

Since matrix-vector products with a Toeplitz matrices can be carried
out in $\OOM(n \log n)$ operations, the power method for estimating
$\norm{T}_2$ can be implemented with the same complexity per iteration.
Alternatively, $\norm{T}_1$ can be computed at little cost.
\begin{lemma}
Let $T \in \C^{n \times n}$ be a Toeplitz matrix. Then $\norm{T}_1$ can be
computed in $\OOM(n)$ operations.
\end{lemma}
\begin{proof}
Let us denote the first column and row of $T$ by $c$ and $r$, respectively,
and set $\mu_j \defby \norm{T e_j}_1$, $1 \le j \le n$.  From the
structure of $T$ we find that
\begin{equation*}
    \mu_1 = \sum_{i=1}^n |c_i|, \quad
        \mu_{j+1} = \mu_j - |c_{n - j + 1}| + |r_j|
        \; \text{for $1 \le j < n$},
\end{equation*}
and hence $\norm{T}_1 = \max_{1 \le j \le n} \{\mu_j\}$ can be
computed in $\OOM(n)$.
\end{proof}

Once the scaling parameter $\rho$ is determined, the generator of $T$
is scaled accordingly, which obviously requires only $\OOM(n)$ operations.

\subsection{Fast solution of Toeplitz and Toeplitz-like systems}
\label{sec:tl_solve}

In order to compute generators for rational functions of Toeplitz
matrices, we need to solve linear systems of equations with
Toeplitz and Toeplitz-like matrices; see Theorem~\ref{thm:drank_r(T)}.
We will now briefly summarize a
well established technique for the solution of such systems in
quadratic time (the ``GKO algorithm''~\cite{Gohberg1995}).

Let $A \in \C^{n \times n}$ be a Toeplitz-like matrix of displacement rank $r \ll n$, so that $A$
satisfies the matrix Stein equation~\eqref{eq:stein} with a low-rank
right hand side.  It is well known (see, e.g.,~\cite[sec.~0.2]{Gohberg1995} and the
references therein) that $T$ also satisfies numerous other matrix equations, including
the Sylvester equation
\begin{equation}
\label{eq:tl_sylvester}
    \Delta_{Z_1,Z_{-1}}(T) \defby Z_1 T - T Z_{-1} = \tilde G \tilde B^*,
\end{equation}
with low-rank right-hand side and $Z_\delta \defby Z + \delta e_1 e_n^*$.

One could directly apply the generalized Schur
algorithm~\cite{Kailath1995} to either representation~\eqref{eq:stein}
or~\eqref{eq:tl_sylvester} in order to solve linear systems with $A$ in $\OOM(rn^2)$ operations,
but without further assumptions on $A$, such as well-conditioned leading principal submatrices, or more
involved algorithmic techniques~\cite{Chandrasekaran1998} a numerically stable
solution is not guaranteed.

Instead, we propose to use a transformation~\cite[prop.~3.1]{Gohberg1995} (see
also~\cite{Heinig1995})
of~\eqref{eq:tl_sylvester} to a \emph{Cauchy-like} Sylvester
displacement equation
\begin{equation}
\label{eq:cl_sylvester}
    D_1 C  - C D_2 = \hat{G} \hat{B}^*,
\end{equation}
with the same displacement rank and where $D_1, D_2$ are
diagonal matrices. The transformation
between~\eqref{eq:tl_sylvester} and~\eqref{eq:cl_sylvester} involves
only FFTs and diagonal scalings.  The fact that the Sylvester
operator matrices $D_1$ and $D_2$ are now diagonal allows for pivoting
within the generalized Schur algorithm, requiring in total
$\OOM(r n^2)$ operations.  Combined with further safeguarding techniques
one obtains an efficient algorithm for the solution of linear systems
with $A$ that enjoys similar stability properties as traditional
Gaussian elimination with pivoting~\cite{Gu1998}.

For our purpose of evaluating rational matrix functions of Toeplitz
matrices only one minor technicality needs to be resolved:  The
generator matrices $G,B$ with respect to the matrix Stein
equation~\eqref{eq:stein} need to be transformed to generator matrices
with respect to the Sylvester equation~\eqref{eq:tl_sylvester}. The following lemma shows that the corresponding displacement rank increases at most by two.

\begin{lemma} \label{lem:converttoeplitz}
Let $A \in \C^{n \times n}$ be a Toeplitz-like matrix, and let $\smb
c\\\alpha \sme$ and $\smb r & \alpha\sme$ denote its last column and row,
respectively.  If $T$ denotes the Toeplitz matrix
with first column $\smb \alpha \\ c \sme$ and first row $\smb \alpha
& r\sme$ then
\begin{equation*}
    \Delta_{Z_1, Z_{-1}}(A) = (\nabla(T) - \nabla(A)) Z_{-1}.
\end{equation*}
\end{lemma}
\begin{proof} One directly calculates that
\begin{align*}
\Delta_{Z_1, Z_{-1}}(A) Z_{-1}^*
& = Z_1 A Z_{-1}^* - A
= (Z + e_1 e_n^*) A (Z^* - e_n e_1^*) - A \\
& = Z A Z^* - A + e_1 e_n^* A Z^* - Z A e_n e_1^* - e_1 e_n^* A e_n e_1^*\\
& = - \nabla(A) +
    \begin{bmatrix}
        \alpha & r\\
        c      & 0_{n-1,n-1}
    \end{bmatrix}
= - \nabla(A) + \nabla(T).
\end{align*}
\end{proof}

To compute a $\Delta_{Z_1,Z_{-1}}$ generator for $A$ using Lemma~\ref{lem:converttoeplitz}, one needs to reconstruct the last column and row of $A$ from a generator with respect to $\nabla$. According to~\eqref{eq:disptoeplitz} this requires $2r$ matrix-vector multiplications with triangular Toeplitz matrices and can hence be computed in $\OOM(r n \log n)$ operations. 
We can summarize the preceding discussion as follows.

\begin{corollary}
Let $A \in \C^{n \times n}$ be a Toeplitz-like matrix of displacement rank $r$. Then linear systems with $A$ can
be solved in $\OOM(rn^2)$ operations.
\end{corollary}

\subsection{Computing generators of Toeplitz matrix polynomials}
\label{sec:polyval}

In Lemma~\ref{lem:power_generators} we computed explicit expressions for
generators of monomials $T, T^2, T^3, \dotsc$ .  Within these
expressions, one needs to (repeatedly) apply the matrices
\begin{equation*}
    (Z - I) T (Z - I)^{-1} \quad \text{and} \quad
    (Z - I) T^* (Z - I)^{-1}
\end{equation*}
to a given canonical Toeplitz generator $(G,B)$.  Note that applying
$(Z - I)^{-1}$ to a vector amounts simply to computing the vector of
its cumulative sums, and that the application of $Z - I$ to a vector
can be evaluated with $n-1$ subtractions.  Hence, both operations
require $\OOM(n)$ operations.  Finally, as mentioned in
Section~\ref{sec:toeplitz}, matrix-vector products with $T$ and $T^*$
can be evaluated in $\OOM(n \log n)$ operations, so that we have the
following result.

\begin{corollary}
Let $T \in \R^{n,n}$ be a Toeplitz matrix, then a set of generators
for the monomials $T, T^2, \dotsc, T^s$ can be computed with $\OOM(sn
\log n)$ operations.
\end{corollary}

Because of the nested structure of the monomial generators
(cf.~\eqref{eq:power_update_G}--\eqref{eq:power_update_B}), only the
generator $(G_s, B_s)$ for the leading monomial $T^s$ is actually
needed for the evaluation of $p(T) \defby \sum_{k=0}^s a_k T^k$.  A
generator for $p(T)$ can be computed by appropriate linear combination
of the block columns of $G_s$, i.e., there exists a matrix $X \in
\R^{3s-1,3s-1}$ defined through the coefficients $a_0, \dotsc, a_s$,
such that $(G_s X, B_s)$ is a generator for $p(T)$.  For example, if we set
\begin{equation*}
    X =
    \begin{bmatrix}
        a_2 I_2 & 0       & 0\\
        0       & a_2 I_2 & 0\\
        0       & 0       & a_2\\
    \end{bmatrix}
    +
    \begin{bmatrix}
        0       & a_1 I_2 & 0\\
        0       & 0       & 0\\
        0       & 0       & 0\\
    \end{bmatrix},
\end{equation*}
then $(G_2 X, B_2)$ is a generator for $a_2 T^2 + a_1 T$.

Alternatively, a Horner-like scheme can be used to
compute a generator for $p(T)$.  Let $T_k$, $0\le k \le s$ be the
$k$th Horner polynomial, defined via the recursion
\begin{equation*}
    T_0 \defby a_s I, \quad
        T_{k} \defby T T_{k-1} + a_{s-k} I \: \text{for $1 \le k \le s$},
\end{equation*}
then $T_k$ is the Schur complement of $-I$ in the embedding
\begin{equation}
\label{eq:horner_embedding}
    M =
    \begin{bmatrix}
        -I    & T_{k-1} \\
         T    & a_k I
    \end{bmatrix}.
\end{equation}
Using similar arguments and techniques as for the evaluation of
monomials in $T$, it follows that the evaluation of $p(T)$ based
on~\eqref{eq:horner_embedding} can be carried out in $\OOM(sn \log n)$
operations.  In contrast to the evaluation based on monomials of $T$,
the resulting generator has length $2s$.

\subsection{Evaluating rational approximants by solving Toeplitz-like
systems}
\label{sec:rat_eval}

We now turn to the computation of generators for rational functions of
$T$. Let $r(z) = \frac{p(z)}{q(z)}$ be a rational function, and let
$(G_p, B_p)$ and $(G_q, B_q)$ be generators for $p(T)$ and $q(T)$,
respectively; see Section~\ref{sec:polyval}. From equations~\eqref{eq:rat_gen_G}--\eqref{eq:rat_gen_B}, we see 
that a generator for $r(T) = q(T)^{-1} p(T)$ is found by solving the linear
Toeplitz-like systems
\begin{equation}
\label{eq:ratsolvesystems}
    q(T)^{-1} (Z-I)^{-1}
    \begin{bmatrix}
        G_q & G_p
    \end{bmatrix}
    \quad
    \text{and}
    \quad
    q(T)^{-*} (Z-I)^{-1}
    \begin{bmatrix}
        B_q & e_1
    \end{bmatrix}.
\end{equation}
In total there are $2 \deg(q) + \deg(p) + 1$ right hand sides to solve
for, and since the displacement rank of $q(T)$ is at most $2 \deg(q)$,
the techniques outlined in Section~\ref{sec:tl_solve} yield the
following result.

\begin{corollary}
Let $T \in \C^{n\times n}$ be a Toeplitz matrix, and $r(z) =
\frac{p(z)}{q(z)}$ a rational function of degree $s = \max\{\deg(p),
\deg(q)\}$.  Then a generator for $r(T) = q(T)^{-1} p(T)$ can be
computed with $\OOM(s^2n^2)$ operations.
\end{corollary}

The dependence on $s^2$ in the above statement is nearly negligible in our
context, since the degrees of the Pad\'e approximants we will be using
are typically small, and never larger than thirteen.  Note also that
\eqref{eq:rat_gen_G}--\eqref{eq:rat_gen_B} involve matrix-vector
multiplications with $p(T)$ and $p(T^*)$, but the cost for these are
dominated by solving the linear systems~\eqref{eq:ratsolvesystems}.

\subsection{Evaluating rational approximants by partial fraction expansion}
\label{sec:pf_expansion}

Any rational function with simple poles can be expressed as a partial fraction
expansion
\begin{equation*}
    r(z) = \sum_{i=1}^m \frac{\beta_i}{z - \alpha_i} + p(z),
\end{equation*}
where $m$ is the number of poles $\alpha_i$ with residues $\beta_i$,
and $p$ is a polynomial.  Let $T \in \C^{n\times n}$ be a Toeplitz matrix,
and assume that none of the poles of $r$ is an eigenvalue of $T$.
Generators for $p(T)$ can be computed using the techniques described
in Section~\ref{sec:polyval}, and we now discuss the computation of
generators for the Toeplitz-like matrix 
\begin{equation}
\label{eq:pf_expansion}
    \sum_{i=1}^{m} \beta_i (T - \alpha_i I)^{-1}.
\end{equation}

First note that for any $\alpha \in \C$  the matrix $T_\alpha = T +
\alpha I$ is a Toepliz matrix as well, with the generator 
\begin{equation*}
G_\alpha =
\begin{bmatrix}
    t_0 + \alpha & 1\\
    t_1 & 0\\
    \vdots & \vdots\\
    t_{n-1} & 0
\end{bmatrix},
\quad
B_\alpha =
\begin{bmatrix}
    1 & 0\\
    0 & t_{-1}\\
    \vdots & \vdots\\
    0 & t_{-n+1}
\end{bmatrix}.
\end{equation*}
Hence, the evaluation of~\eqref{eq:pf_expansion} is simply the sum of
$m$ inverse Toeplitz matrices and we can therefore apply the result from Remark~\ref{remark:toeplitzinverse} and the related \emph{Gohberg-Semuncul formulas} (see, e.g.,~\cite{Huckle1998} for an overview) to compute a generator for~\eqref{eq:pf_expansion} by solving $\OOM(m)$ linear Toeplitz systems using the technique described in Section~\ref{sec:tl_solve}.

In the common case where $T$ is real, and the
expansion~\eqref{eq:pf_expansion} involves pairs of complex conjugates
shifts $\alpha, \bar{\alpha}$ and residues $\beta, \bar{\beta}$, then
\begin{equation*}
    \bar{\beta} (T - \bar{\alpha})^{-1} = \overline{\beta (T - \alpha)^{-1}},
\end{equation*}
so that a real generator for $\beta(T - \alpha)^{-1} + \bar{\beta} (T
- \bar{\alpha})^{-1}$ can be computed by means of solving two linear
  Toeplitz systems (instead of four).  So let $(G,B)$ be a generator for
$\beta(T - \alpha)^{-1}$, then
\begin{align*}
    & \nabla( \beta(T - \alpha)^{-1} + \bar{\beta} (T - \bar{\alpha})^{-1} )
    = \nabla( \beta(T - \alpha)^{-1} ) + \overline{\nabla( \beta(T - \alpha)^{-1} )}\\
    & = GB^* + \overline{GB^*} = 2 \real(G B^*)
    = 2 (\real(G) \real(B)^* + \imag(G) \imag(B)^*).
\end{align*}
In the case of a Pad\'e approximant, this implies that the number of
Toeplitz matrix inversions is roughly halved.  Of course, the
asymptotic cost is unchanged, and we summarize our findings as
follows.
\begin{corollary}
Let $T \in \C^{n \times n}$ be a Toeplitz matrix, and $r(z) = \sum_{i=1}^m
\beta_i (z - \alpha_i)^{-1}$ a rational function.  Then $r(T)$ can be
evaluated in $\OOM(m n^2)$ operations involving at most $2m+2$
solutions of linear Toeplitz systems.
\end{corollary}

\subsection{Iterative squaring}
\label{sec:squaring}

At the final stage of scaling and squaring algorithms we have at hand
a rational approximation $r(2^{-\rho}T) \approx \exp(2^{-\rho} T)$,
and in order to obtain an approximation for $\exp(T)$, the initial
scaling is undone by squaring the matrix $r(2^{-\rho}T)$ $\rho$ times.

Let $(G, B)$ be a generator of length $s$ for the Toeplitz-like matrix
$A$.  Then a generator of length $2s+1$ for $A^2$ can be computed
using Theorem~\ref{thm:product_generators}.  The cost for computing
the new generator is dominated by the evaluation of the products $AG$
and $A^*B$.  Each of these products can be computed based on the
expansion~\eqref{eq:disptoeplitz}, involving $2s^2$ multiplications
with triangular Toeplitz matrices, resulting in an operation of
complexity $\OOM(s^2n \log n)$.

Each squaring operation effectively doubles the length of the
generator matrices, and hence after $\rho$ squaring operations we
would obtain generator matrices of length $\OOM(2^\rho s)$, which is
computationally feasible only for tiny values of $\rho$.
However, if the spectrum of $T$ allows for low degree rational
approximation of $\exp(T)$ (see Sec.~\ref{sec:rationalapproximation}), the
same is true for each scaled matrix $2^{-k}T$, $1 \le k \le \rho$.
Consequently, if the rational approximation to $\exp(2^{-\rho}T)$ is
such that
\begin{equation}
\label{eqn:intermed_ranks}
r(2^{-\rho}T)^{2^k} \approx \exp(2^{\rho-k}T),\quad 0 \le k \le \rho,
\end{equation}
then Corollary~\ref{cor:approx_exp} shows that each displacement $\nabla(r(2^{-\rho}T)^{2^k})$ is close to
a low rank matrix, and a generator compression
(Alg.~\ref{alg:gencompress}) applied after every squaring operation
will reduce the intermediate generator length back to $\OOM(s)$,
without compromising the approximation quality of the final
approximation to $\exp(T)$.  The computational cost for each of this
compressions is dominated by the matrix multiplications $AG$ and
$A^{*} B$. The following Corollary summarizes the discussion.

\begin{corollary}
Generators for the sequence $r(2^{-\rho}T)^2, \dotsc,
r(2^{-\rho}T)^{2^\rho}$ can be computed in $\OOM(\rho s^2 n \log n)$,
provided that each intermediate displacement
$\nabla(r(2^{-\rho}T)^2)$, $\dotsc$, $\nabla(r(2^{-\rho}T)^{2^\rho})$
has numerical rank $\OOM(s)$.
\end{corollary}

\subsection{Reconstruction of Toeplitz-like matrices}
\label{sec:reconstruction}

As mentioned in Section~\ref{sec:toeplitz}, a Toeplitz-like matrix can
be reconstructed from a generator based on~\eqref{eq:reconstruction}.
We note next that this operation can be implemented efficiently.
\begin{lemma}
    \label{lem:alg_reconstruction}
Let $A \in \C^{n\times n}$ be a Toeplitz-like matrix, and 
$(G,B)$ a generator for $T$ of length $r$.  Then $A$ can be
computed from $(G,B)$ in $\OOM(rn^2)$ operations.
\end{lemma}
\begin{proof}
The number of operations for computing $D = GB^*$ is $\OOM(rn^2)$.
Then the expression
\begin{equation*}
    A = \sum_{k=0}^{n-1} Z^k D (Z^*)^k
\end{equation*}
can be evaluated in $\OOM(n^2)$ operations by noting that the $k$th
subdiagonal (superdiagonal) of $A$ is just the vector of cumulative
sums of the $k$th subdiagonal (superdiagonal) of $D$.
\end{proof}

\begin{remark} If only the diagonal of $A$ is of interest then the
    proof of Lemma~\ref{lem:alg_reconstruction} shows that this
    diagonal can be reconstructed by forming the cumulative sum of the
    vector $d = \diag(GB^*)$. The arithmetic cost for forming $d$ and thus the entire cost of extracting the diagonal of $A$ is $\OOM(nr)$.
    In fact any
    banded section of $A$ can be reconstructed by only computing the
    corresponding banded section of $GB^*$.
\end{remark}

\subsection{A remark on the use of the FFT}

Many of the operations discussed in this section involve or even
reduce to matrix-vector multiplication with Toeplitz and Toeplitz-like
matrices, cf.\@ Secs.~\ref{sec:toep_norm}, \ref{sec:polyval},
\ref{sec:rat_eval}, \ref{sec:squaring}.  The complexity of computing a
matrix-vector product $Ax$ for a Toeplitz-like matrix $A \in \C^{n,n}$
of displacement rank $r$ is $r n \log n$.  Although carrying out these
multiplications using the FFT is \emph{asymptotically} faster than
standard matrix-vector multiplication, an actual computational
advantage is gained only for sufficiently large matrix dimension $n$.

If that is not the case, it is preferable to resort to standard
matrix-vector multiplication.   Since $A$ is of displacement rank $r$,
the reconstruction of $A$ from a generator can be done in $\OOM(rn^2)$
operations, as described in Sec.~\ref{sec:reconstruction}.
Consequently, a single matrix-vector product can be computed using
standard multiplication in $\OOM(rn^2 + n^2)$ operations, and $\ell$
matrix-vector products with $A$ can be evaluated in $\OOM( (r+\ell)
n^2)$ operations.  Note that in this latter case, the cost for FFT based
multiplication is $\OOM(\ell r n \log n)$.

\section{Scaling and squaring algorithms for Toeplitz matrices}
\label{sec:s_and_s}

In Section~\ref{sec:rationalapproximation} we have shown that rational
approximations to the matrix exponential of a Toeplitz matrix $T$
enjoy low displacement rank, provided that $T$ is
negative real or sectorial.   We will now use scaling and
squaring algorithms that take advantage of this property.  Based on
the techniques presented in Section~\ref{sec:algtools}, the resulting
algorithms will require $\OOM(n^2)$ operations for computing
$\exp(T)$, which is optimal since the output is also of size $n^2$.

Denote by $r_{k,m}(z) = \frac{p_{k,m}(z)}{q_{k,m}(z)}$ the
$(k,m)$-Pad\'e approximant to the exponential function, meaning that
the numerator polynomial is of degree $k$, and the denominator
polynomial of degree $m$. Scaling and squaring algorithms take
advantage of the fact that Pad\'e approximations are very accurate
close to origin.  An input matrix $A$ is thus scaled by a power of
two, so that $\norm{2^{-\rho}A} \lessapprox 1$, and then the Pad\'e
approximant $r_{k,m}(2^{-\rho}A)$ is computed.  Finally, an
approximation to $\exp(A)$ is obtained by squaring the result repeatedly,
viz.
\begin{equation*}
    \exp(A) \approx  r_{k,m}(2^{-\rho}A)^{2^{\rho}},
\end{equation*}
using the identity $\exp(A) = \exp(\sigma^{-1}A)^\sigma$, $\sigma \in
\C\setminus\{0\}$.

Different strategies for choosing the scaling parameter $\rho$ and the
Pad\'e degree $(k,m)$ yield different methods.  We will discuss two
recently proposed scaling and squaring methods.  The first one,
described by Higham~\cite{Higham2009}, is based on a diagonal Pad\'e
approximation of degree at most 13 and makes no assumption on the
spectrum of the input matrix.  The second one by G\"uttel and
Nakatsukasa~\cite{Guttel2016} employs a subdiagonal Pad\'e
approximation of much smaller degree, and is particularly useful if
the imaginary parts of the rightmost eigenvalues do not vary too much.
Both scaling and
squaring methods have been shown to behave in a forward stable manner
for normal matrices.

\subsection{A diagonal scaling and squaring method}
\label{sec:diag_expm}

The scaling and squaring method designed by Higham~\cite{Higham2009}
was until recently the default method to compute the matrix
exponential in {\textsc Matlab}, available via the command {\ttfamily
expm}.  It scales the input matrix $A$ so
that $\norm{2^{-\rho}A}_1 \lessapprox 5.4$, and approximates $\exp(A)$ using
a diagonal Pad\'e approximation, i.e., $k=m$ in the notation from
above.  The approximation degree is at most $13$, or less for matrices
that need not be scaled.  This parameter choice is designed such that
the approximation error can be interpreted as a backward error $E$ (in
any consistent matrix norm)
\begin{equation} \label{eq:expm_bwb}
    r_{m,m}(2^{-\rho}A)^{2^\rho} = \exp(A + E), \;
    \norm{E} \le u \norm{A},
\end{equation}
where $u$ denotes the unit roundoff in double precision; see~\cite[Thm.~2.1]{Higham2009}.  Note that no assumption on the
spectral properties of $A$ have been made.  The matrix $E$ can be
shown to commute with $A$, and hence one obtains immediately
\begin{align} \label{eq:expm_fwb}
\begin{split}
    \unorm{r_{m,m}(2^{-\rho}A)^{2^\rho} - \exp(A)}
    &\le \unorm{\exp(A) ( \exp(E) - I )} \\
    &\le \unorm{\exp(A)} \unorm{E} \unorm{\exp(E)}
    \le \unorm{\exp(A)} u \norm{A} e^{u \norm{A}},
\end{split}
\end{align}
which bounds the forward error of the approximation.  At the same
time, Higham shows that the matrix $q_{m,m}(A)$ is well conditioned
under this parameter regime.

We now show for the case of real spectra that the numerical rank is
bounded throughout the squaring phase
(cf.~\eqref{eqn:intermed_ranks}).  By effecting a shift $T \leftarrow
T - \mu I$ we may assume that the spectrum is actually negative real,
which simplifies the notation in what follows.
\begin{lemma}
    Let $T$ be a diagonalizable Toeplitz matrix with spectrum in
    $(-\infty, 0]$.  Set $\tau = u \norm{T} \exp(u\norm{T})$, and 
    \begin{equation*}
        s = \left\lceil \frac{\log(\tilde{C}) - \log(\tau)}{\log(V)}
        \right\rceil,
    \end{equation*}
    with $\tilde{C}$ and $V$ as in Corollary~\ref{cor:approxhnd} (we
    may safely assume that $s \ge 0$).
    Then $r_{m,m}(2^{-\rho} A)^{2^\sigma}$ and all the intermediate
    matrices in the squaring phase have $2\tau$-displacement rank
    $2s+1$.
\end{lemma}
\begin{proof}
We abbreviate $r \defby r_{m,m}$. The squaring phase involves the quantities
\begin{equation*}
    r(2^{-\sigma} T),
    r(2^{-\sigma} T)^{2^1},
    r(2^{-\sigma} T)^{2^2},
    \dotsc,
    r(2^{-\sigma} T)^{2^\sigma},
\end{equation*}
and we proceed by showing that each of these powers is close to the matrix
exponential of the matrix $T_k := 2^{-k} T$. Regarding $T_k$ as the \emph{input matrix} to
{\ttfamily expmt}, the
scaling power selected by the algorithm design is $\sigma - k$, and
the approximant $r$ is the same for all $k$. So
we obtain from~\eqref{eq:expm_bwb} for $0 \le k \le \sigma$
\begin{equation*}
    r(2^{-\sigma+k} 2^{-k}T)^{2^k} = r(2^{-\sigma}T)^{2^k}
    = \exp(T_k + E_k), \quad
    \norm{E_k} \le 2^{-k} u \norm{T}.
\end{equation*}

A forward error bound for the approximation of the intermediate
exponentials, that is, for each squaring iteration $k$, follows
from~\eqref{eq:expm_fwb}, 
\begin{align*}
    \norm{r(2^{-\sigma} T)^{2^k} - \exp(T_k)}
    & = \norm{ \exp(T_k + E_k) - \exp(T_k) }\\
    & \le u 2^{-k} \norm{T} \exp( u 2^{-k} \norm{T} )
    \le \tau.
\end{align*}
Since all the matrices $T_k$ (trivially) have negative real
spectra, the result now follows from
Corollary~\ref{cor:approx_exp}.
\end{proof}
Note that the non-normality of $T$ is reflected in the
constant $\tilde{C}$, via the condition number of an eigenbasis for
$T$.  If $T$ arises from a discretization process,
the lemma shows that the $2\tau$-numerical displacement rank remains bounded
as long as this condition number is bounded.  For highly non-normal
matrices, the usefulness of our result is limited.  We remark, however,
that the behaviour of scaling and squaring, in particular its
stability, is not
well understood in this case.

We now leave the regime of negative real spectra, but continue to
assume that $T$ is not highly non-normal, that is, $T$ can be
diagonalized by a reasonably well-conditioned matrix.  The preceding
discussion in fact applies to \emph{any} Toeplitz matrix $T$ whose
spectrum is contained in a subset of the complex plane, where the
error of the rational best approximation to $e^z$ decays exponentially
in the approximation degree (the involved constants change, however).
Another well known example for this situation are sectorial matrices
(see Section~\ref{sec:approxexponential}).

In the absence of such a rational best approximation result, which are in
general quite difficult to obtain, the displacement rank of $\exp(T)$, as well
as the intermediates in the squaring phase, are not \emph{bounded a priori}
and independently of the scaling power $\rho$.  However the
displacemnt ranks may still remain small, see
Section~\ref{sec:random} for an example.

\begin{algorithm}[t]
\caption{{\ttfamily expmt} -- Diagonal scaling \& squaring
from~\cite{Higham2009} for Toeplitz matrices}
\label{alg:expmt}
\begin{algorithmic}[1]
\REQUIRE Toeplitz matrix $T\in \C^{n\times n}$, given by its first column
$c$ and row $r$.
\ENSURE Generator $(G,B)$ such that $\toepop(G,B) \approx \exp(T)$
\STATE{Compute $\norm{T}_1$} \hfill
    \COMMENT{$\OOM(n)$, Sec.~\ref{sec:toep_norm}}
\STATE{Chose scaling parameter $\rho$ and Pad\'e approximant
    $r_m(z) = \tfrac{p_m(z)}{q_m(z)}$} \hfill \COMMENT{\cite{Higham2009}}
\STATE{Scale $c \leftarrow 2^{-\rho} c$, $r \leftarrow 2^{-\rho} r$} \hfill
    \COMMENT{$\OOM(n)$}
\STATE{$(G_p, B_p) \leftarrow$ generator for $p_m(T)$} \hfill
    \COMMENT{$\OOM(mn \log n)$, Sec.~\ref{sec:polyval}}
\STATE{$(G_q, B_q) \leftarrow$ generator for $q_m(T)$} \hfill
    \COMMENT{$\OOM(mn \log n)$, Sec.~\ref{sec:polyval}}
\STATE{$(G,B) \leftarrow$ generator for $r_m(T) = q_m(T)^{-1} p_m(T)$} \hfill
    \COMMENT{$\OOM(m^2n^2)$, Sec.~\ref{sec:rat_eval}}
\FOR{$k=1$ \TO $\rho$} \label{line:square_beg}
\STATE  $(\tilde{G}, \tilde{B}) \leftarrow$ generator for $\toepop(G,B)^2$  \hfill
    \COMMENT{$\OOM(m^2 n \log n)$, Sec.~\ref{sec:squaring}}
\STATE $(G,B) \leftarrow$ compress $(\tilde{G}, \tilde{B})$ \hfill
    \COMMENT{$\OOM(m^2 n)$, Alg.~\ref{alg:gencompress}}
\ENDFOR \label{line:square_end}
\STATE \COMMENT{optionally} reconstruct $\toepop(G,B)$ \hfill
    \COMMENT{$\OOM(mn^2)$, Sec.~\ref{sec:reconstruction}}
\end{algorithmic}
\end{algorithm}

Assuming that the numerical displacement ranks during the squaring
phase are bounded, a practical algorithm of quadratic complexity is
obtained by replacing the unstructured matrix computations used
in~\cite{Higham2009} by their structured counterparts explained in
Sec.~\ref{sec:algtools}.  The resulting method is shown in
Algorithm~\ref{alg:expmt}.

We close the discussion by noting that Algorithm~\ref{alg:expmt}
\emph{almost} achieves our goal of designing a method of quadratic
complexity for the Toeplitz matrix exponential.  While the
approximation degree $m$ is bounded by $13$, the scaling power $\rho$
still grows logarithmically with $\norm{A}_1$, and consequently the
number of squaring iterations is not bounded independently of the
numerical values of $A$.  From a practical point of view
Algorithm~\ref{alg:expmt} still behaves like an algorithm of quadratic
complexity.

\subsection{A subdiagonal scaling and squaring method}
\label{sec:sexpmt}

The second scaling and squaring method we adapted for the Toeplitz case is
described and analyzed in~\cite{Guttel2016}.  In contrast to Higham's design,
it (typically) employs a \emph{subdiagonal} Pad\'e approximation (hence
``{\ttfamily sexpm}''), which is appropriate if the spectrum is located on the
negative real line, or close to it.  If the input matrix $T$ does not have this
property, say, the eigenvalues are only real, this subdiagonal approximation
may be applied to the shifted matrix
\begin{equation*}
    T - \lambda_{\text{max}} I,
\end{equation*}
where $\lambda_{\text{max}}$ denotes the largest eigenvalue of $T$.
(The approximation to $\exp(T)$ is then recovered by multiplication
with $e^{\lambda_{\text{max}}}$).

Put more generally, it suffices that the rightmost eigenvalues of $T$ do not have
widely varying imaginary parts for this subdiagonal approximation to be very
accurate. We refer to~\cite{Guttel2016} for a complete discussion of this
shifting technique, and the quality of the obtained approximation.

Compared to
{\ttfamily expm}, {\ttfamily sexpm} has several attractive features:
\begin{enumerate}

    \item For the Pad\'e degree $(k,m)$ we have $k,m \le 5$, resulting in
        computational cost savings over the diagonal approximation from the
        previous section.

    \item The Pad\'e approximant can be stably evaluated as a partial fraction
        expansion.  Hence the rational approximation $q_{k,m}(T)^{-1} p_{k,m}$
        involves only solves with Toeplitz matrices instead of Toeplitz-like
        matrices.

    \item The number of scaling iterations $\rho$ is bounded by four,
        \emph{independently} of $A$, implying that the generator length can
        increase at most by a factor of $2^4 = 16$ during the squaring phase.
        (In light of the results from Section~\ref{sec:approxexponential},
        this factor likely still is a gross overestimate, however).  Hence, the
        displacement rank of the approximation to the exponential is bounded
        independently of $A$ as well.  In other words, the potential
        displacement rank growth as discussed in Section~\ref{sec:squaring} is
        not an issue for this method.

\end{enumerate}

If $A \in \C^{n,n}$ is normal, the approximation $B$ of $\exp(A)$
obtained through {\ttfamily sexpm} satisfies
\begin{equation*}
    \norm{B - \exp(A)}_2 \le u \norm{A}_2,
\end{equation*}
and the authors in fact show that their method is a forward stable
method~\cite[Thm.~4.1]{Guttel2016}.  The adaption to the Toeplitz
case, coined {\ttfamily sexpmt}, is shown in Algorithm~\ref{alg:sexpmt}.  Since the
number of squaring iterations is bounded independently of $n$ and $\norm{A}_2$,
it follows that the complexity of Algorithm~\ref{alg:sexpmt}
is $\OOM(n^2)$.

\begin{algorithm}[t]
\caption{{\ttfamily sexpmt} -- Subdiagonal scaling \& squaring
from~\cite{Guttel2016} for Toeplitz matrices}
\label{alg:sexpmt}
\begin{algorithmic}[1]
\REQUIRE Toeplitz matrix $T\in \C^{n\times n}$ with the spectral properties described
    in Section~\ref{sec:sexpmt}, given by its first column
$c$ and row $r$.
\ENSURE Generator $(G,B)$ such that $\toepop(G,B) \approx \exp(T)$
\STATE{Estimate $\norm{T}_2$} \hfill
    \COMMENT{$\OOM(n \log n)$, Sec.~\ref{sec:toep_norm}}
\STATE{Chose scaling $2^\rho$ and Pad\'e approximant
    $r_{k,m}(z) = \sum_{i=1}^m \tfrac{\beta_i}{z - \alpha_j} + p(z)$} \hfill
    \COMMENT{\cite{Guttel2016}}
\STATE{Scale $c \leftarrow 2^{-\rho} c$, $r \leftarrow 2^{-\rho} r$} \hfill
    \COMMENT{$\OOM(n)$}
\STATE Initialize $G = [], B = []$
\FOR{$i=1$ \TO $m$}
\STATE Compute generator $(G_i,B_i)$ for $\beta_i (T - \alpha_iI )^{-1}$ \hfill
    \COMMENT{$\OOM(n^2)$, Sec.~\ref{sec:pf_expansion}}
\STATE $G \leftarrow [G, G_i]$, $B \leftarrow [B, B_i]$
\ENDFOR
\STATE Compute generator for $p(T)$, append to $(G,B)$ \hfill
    \COMMENT{$\OOM(\deg(p) n \log n)$, Sec.~\ref{sec:polyval}}
\FOR{$k=1$ \TO $\rho$}
\STATE  $(\tilde{G}, \tilde{B}) \leftarrow$ generator for $\toepop(G,B)^2$  \hfill
    \COMMENT{$\OOM(m^2 n \log n)$, Sec.~\ref{sec:squaring}}
\STATE $(G,B) \leftarrow$ compress $(\tilde{G}, \tilde{B})$ \hfill
    \COMMENT{$\OOM(m^2 n)$, Alg.~\ref{alg:gencompress}}
\ENDFOR
\STATE \COMMENT{optionally} reconstruct $\toepop(G,B)$ \hfill
    \COMMENT{$\OOM(mn^2)$, Sec.~\ref{sec:reconstruction}}
\end{algorithmic}
\end{algorithm}

\section{Numerical experiments}

We have implemented Algorithms~\ref{alg:expmt} and~\ref{alg:sexpmt} in
\textsc{Matlab}. For solving the Toeplitz-like systems as described in
Section~\ref{sec:tl_solve}, we are using the
drsolve\footnote{\url{http://bugs.unica.it/~gppe/soft/\#drsolve}}
package~\cite{Arico2010}.  All experiments were conducted on a standard
Linux box using a single computational thread.  The {\ttfamily expm}
function of \textsc{Matlab} that we used for various comparisons is
described in~\cite{Higham2008}.

\subsection{Exact error on small matrices} \label{sec:tinycollection}

\begin{figure}[t] \begin{center}
\includegraphics[width=0.49\textwidth]{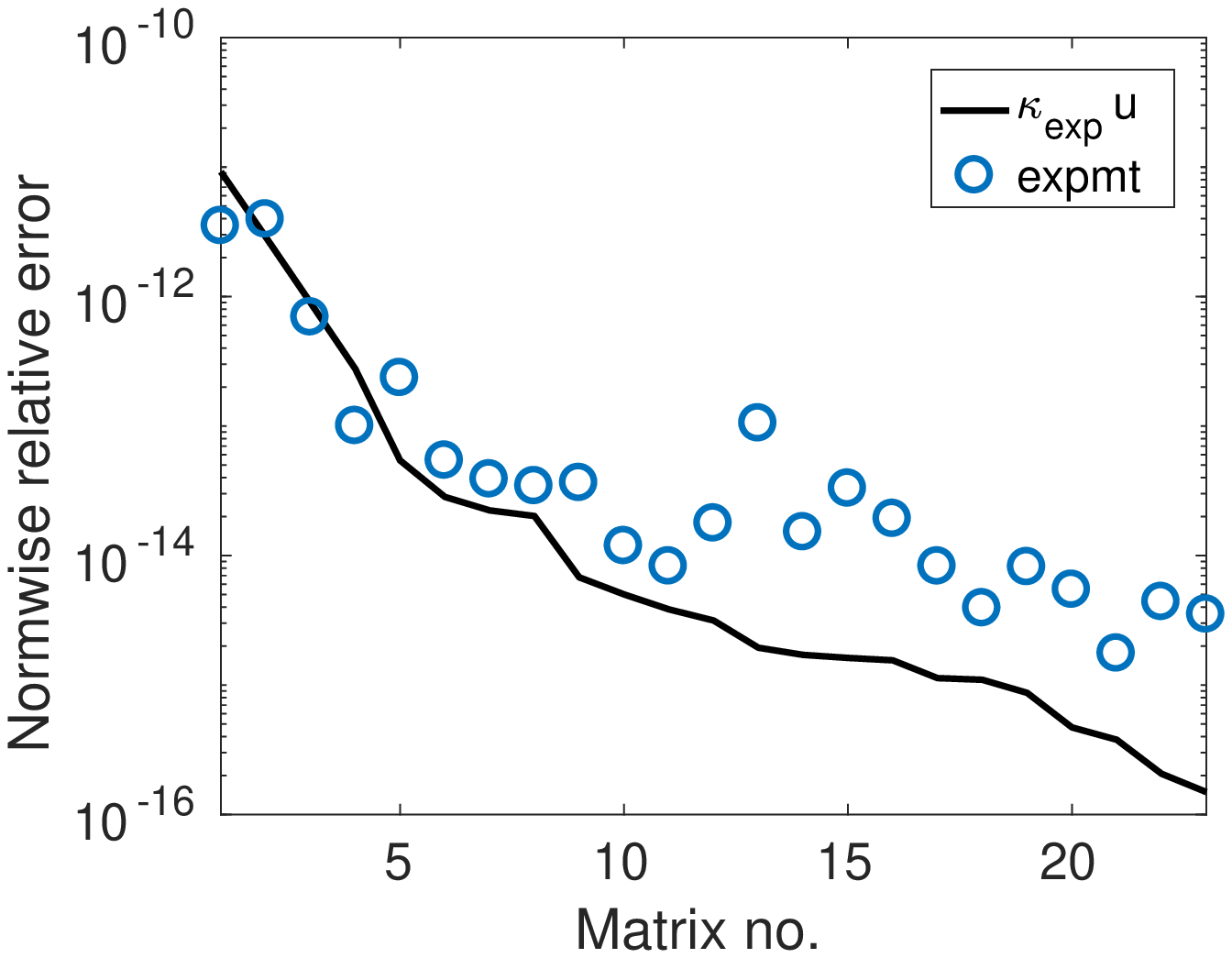}
\hfill
\includegraphics[width=0.49\textwidth]{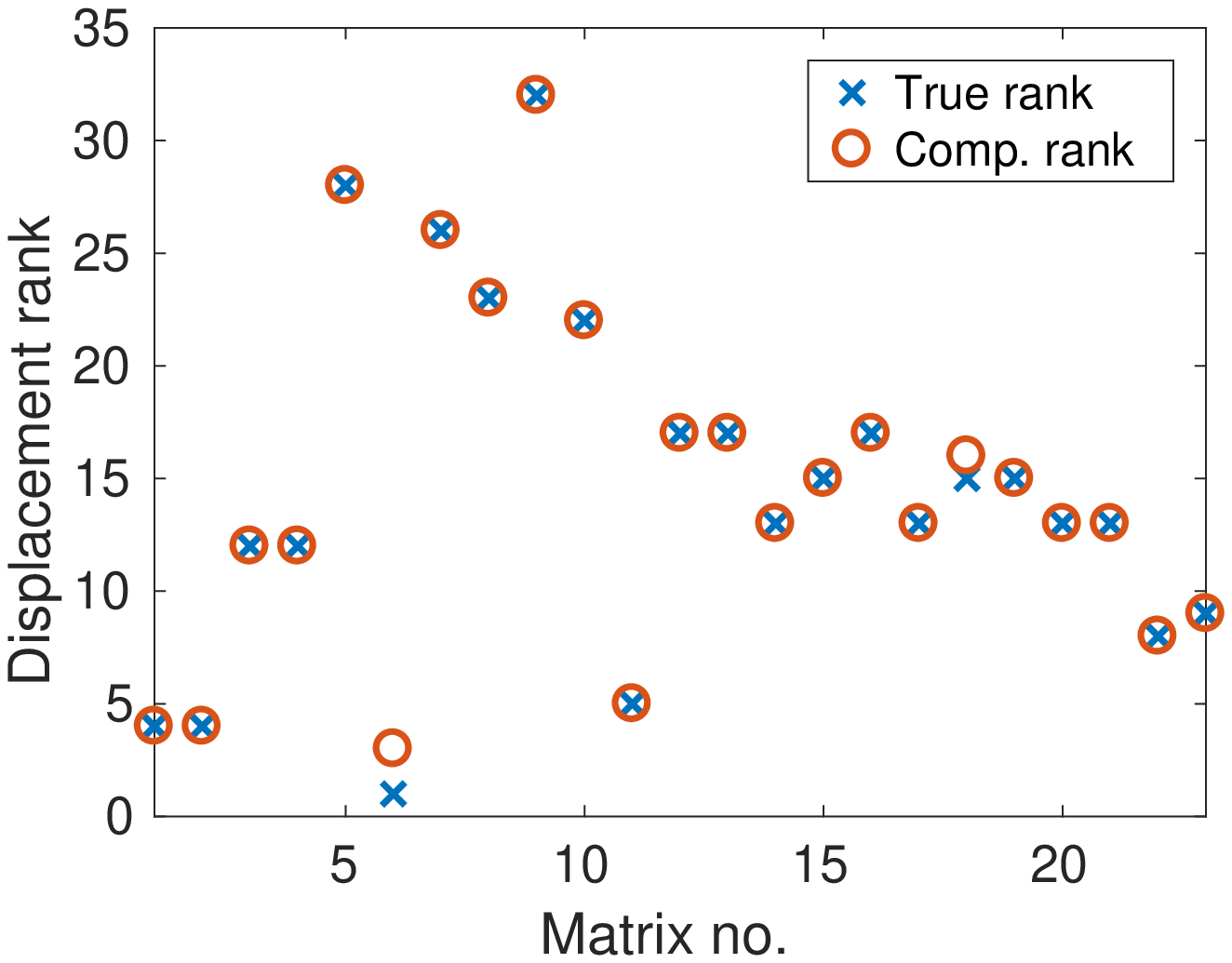}
\end{center} \caption{\emph{Left}: Normwise relative error for the
approximation of $\exp(T)$ via Algorithm~\ref{alg:expmt} for the
matrices listed in Section~\ref{sec:tinycollection}.  The solid black
line indicates the condition numbers (times machine precision).
\emph{Right}: Displacement rank of $\exp(T)$ and the displacement rank
of the approximation.  All matrices have size
$32\times32$.\label{fig:tinycollection}} \end{figure}

As a first test we compute the exact normwise error of the
approximation of $\exp(T)$ via Algorithm~\ref{alg:expmt} on a diverse
set of \emph{small} Toeplitz matrices, from the following sources:
\begin{itemize}

\item The $16$ Toeplitz matrices available via the {\ttfamily
smtgallery} command of the structured matrix
toolbox\footnote{\url{http://bugs.unica.it/~gppe/soft/smt/}}~\cite{Redivo2012}.

\item Seven matrices from~\cite[sec.~5]{Lee2010}.  Specifically, we
generated matrices according to Examples 1 and 2
from~\cite{Lee2010} for time steps $1$, $10$ and $100$ as well
as one instance of Example 3 from~\cite{Lee2010} with time step
$1$. The last example refers to the Merton model, which is
considered further in Section~\ref{sec:merton}.

\end{itemize}

Figure~\ref{fig:tinycollection} shows the normwise relative errors
\begin{equation*} \frac{\norm{\exp(T) -
\expmt(T)}_F}{\norm{\exp(T)}_F}, \end{equation*} where $\expmt(T)$
denotes the computed approximation to $\exp(T)$ obtained from
Algorithm~\ref{alg:expmt}.  The ``exact'' $\exp(T)$ was computed using
\textsc{Matlab}'s variable precision arithmetic with $150$ digits.  Further, we
show for each matrix in the set an approximation to the relative
condition number of the exponential condition
number~\cite[chap.~10]{Higham2008} (black line).  The errors of a
backward stable method would realize errors close to this line, and we
see that the errors of $\expmt$ are roughly bounded by ten times this
quantity.

\subsection{Efficiency of {\ttfamily expmt} outside the sectorial regime}
\label{sec:random}

Let $T$ be a Toeplitz matrix.  In Section~\ref{sec:diag_expm} we
concluded that that the displacement rank of $\exp(T)$, and ranks of the
intermediate approximations~\eqref{eqn:intermed_ranks}, are bounded a
priori only under certain conditions on the spectrum of $T$, e.g., real or
sectorial.  Further, in Example~\ref{ex:oscillation} we showed that
one cannot expect an accurate, low displacement rank approximation of
$\exp(T)$ if all the eigenvalues of $T$ are located on, or close to, the
imaginary axis.

\begin{figure}
    \includegraphics[width=0.49\textwidth]{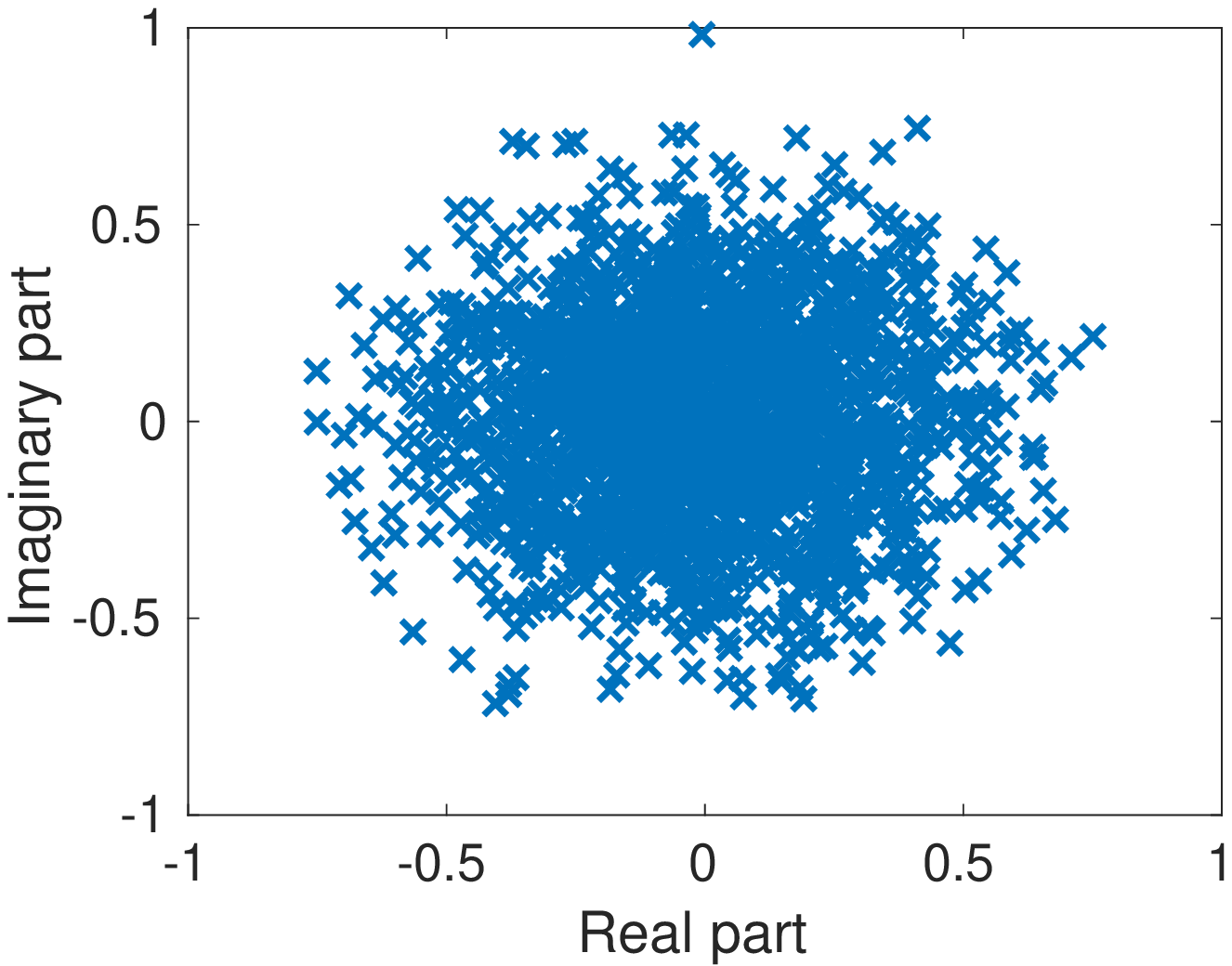}
    \includegraphics[width=0.49\textwidth]{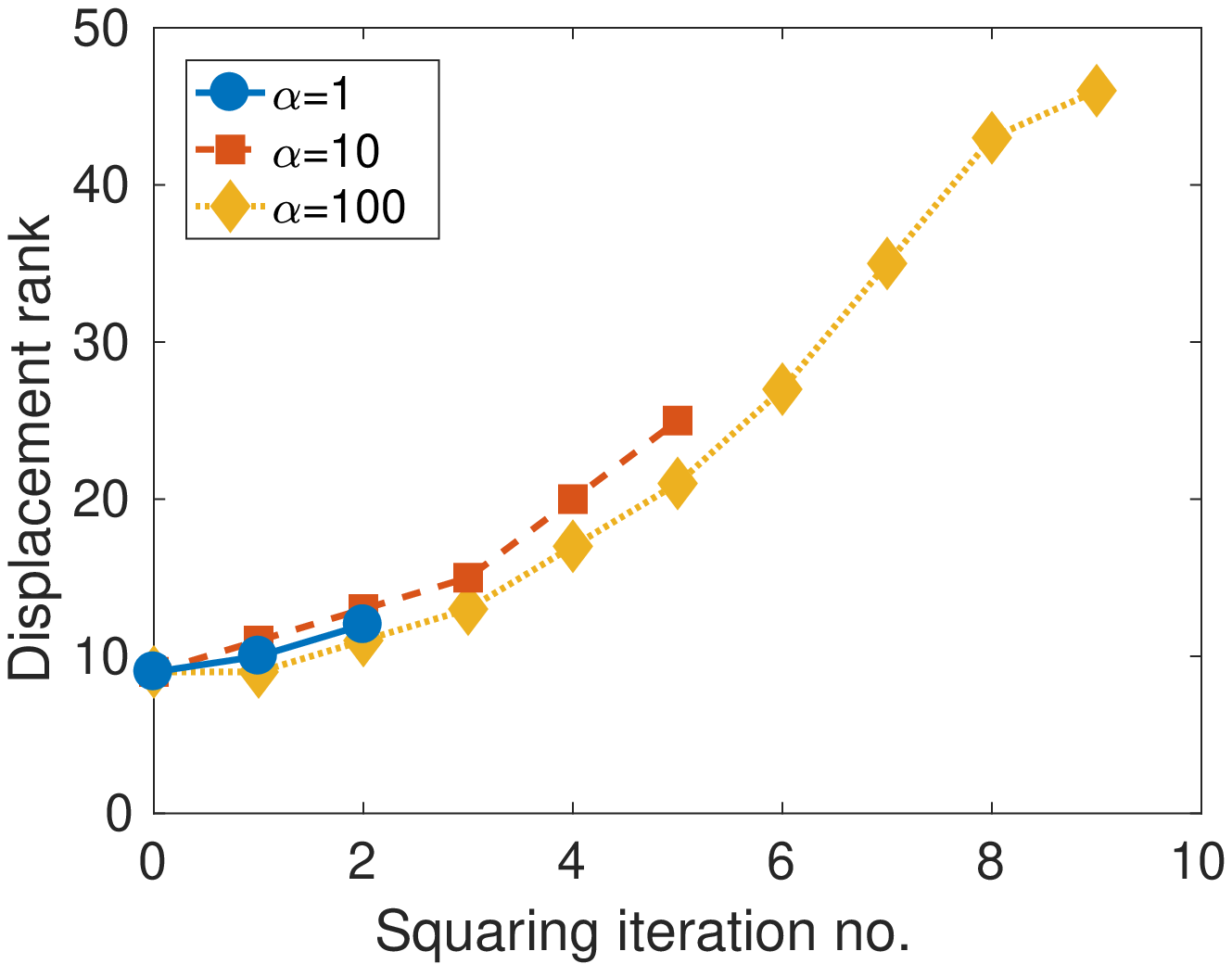}
    \caption{Algorithm~\ref{alg:expmt} applied to $\alpha T$, where
    $T$ is a random matrix and $\alpha \in \{1,10,100\}$.
    \emph{Left}: Spectrum of $T$ in the complex plane. \emph{Right}:
    Displacement rank of the intermediate approximations during the
    squaring phase.  See Section~\ref{sec:random} for a
    discussion.\label{fig:random}}
\end{figure}

This does not imply that Algorithm~\ref{alg:expmt} is necessarily
inefficient if applied to a matrix that is not sectorial.  We
illustrate this setting in Figure~\ref{fig:random}.  There we consider
a Toeplitz matrix $T$ as in~\eqref{eq:toeplitz} where the real and
imaginary parts of each $t_{k}$, $-n+1 \le k \le n-1$ ($n=2000$), have
been drawn from $\mathcal{N}(0,1)$, followed by a scaling so that
$\norm{T}_2 = 1$.  The spectrum of $T$ is shown on the left, and it is
evident that $T$ is not sectorial.

We applied Algorithm~\ref{alg:expmt} to $\alpha T$ for $\alpha \in
\{1, 10, 100\}$, in order to provoke an increasing number of squaring
iterations.  The displacement ranks of the intermediate approximations
over the course of the squaring iterations,
lines~\ref{line:square_beg}--\ref{line:square_end}, are shown on the
right (iteration $0$ corresponds to the initial rational
approximation).  We observe that the displacement rank grows very
slowly in all of the three cases under consideration.  In particular,
the displacement rank does \emph{not double} at each squaring
iteration, which corresponds to the worst case bound discussed in
Section~\ref{sec:squaring}.  We remark that for $\alpha=100$ the run time of
{\ttfamily expmt} was no more than $5$s, while {\ttfamily expm} took more
than $40$s.

Finally we revisit Example~\ref{ex:oscillation}. While
Algorithm~\ref{alg:expmt} will exhibit $\OOM(n^3)$ run time scaling
when applied to matrices of this type, it is still faster than the
corresponding standard {\ttfamily expm} algorithm, because the only
costly operations are the very last squaring steps.  In contrast,
\emph{all} the squaring iterations, as well as the computations
involved in evaluating the rational function, are operations of cubic
complexity in
{\ttfamily expm}.

A comparison of the run times for computing the exponentials in
Example~\ref{ex:oscillation} is given in the following table.

\begin{center}
\begin{tabular}{l|rrr}
    $\alpha$ & {\ttfamily expm} & {\ttfamily expmt} & {\ttfamily expmt}-GEMM\\
    \hline
    1        &  4.55s &  1.96s & 2.01s\\
    10       &  8.04s &  2.92s & 2.98s\\
    100      & 13.15s &  3.28s & 3.41s\\
    1000     & 19.47s & 12.30s & 5.77s\\
\end{tabular}
\end{center}

Comparing the timings of {\ttfamily expm} and {\ttfamily exmpt}, we
find that Algorithm~\ref{alg:expmt} is always faster than
\textsc{Matlab}'s built-in matrix exponential function.  The last
column (``{\ttfamily expmt}-GEMM'') shows the run times for the
following variation of {\ttfamily expmt}:  As soon as the displacement
rank exceeds a certain threshold during the squaring phase (here we
chose $n/6$), we use standard matrix-matrix multiplication instead of
the generator based multiplication (see Section~\ref{sec:squaring}).
With this trivial twist enabled, the run time gains over {\ttfamily
expm} remain quite pronounced even if the exponential is far from
having low numerical displacement rank.

\subsection{Option pricing using the Merton model} \label{sec:merton}

We now turn to the evaluation of
option prices in the Merton model, for one single underlying
asset~\cite{Merton1976}.  There, in contrast to the Black-Scholes
model, the expected return of the asset evolves as a mixture of
continuous and jump processes.  The option value $\omega(\xi, t)$ on
$(-\infty, \infty) \times [0, T]$ satisfies the partial integro-differential equation (PIDE) \begin{equation}
\label{eq:merton} \omega_t = \frac{\nu^2}{2} \omega_{\xi\xi} + \left(
r - \lambda \kappa - \frac{\nu^2}{2}\right) \omega_\xi
    - (r + \lambda) \omega + \lambda \int_{-\infty}^\infty \omega(\xi
      + \eta, t) \phi(\eta) d\eta, \end{equation}
where $T$ denotes time to maturity, $\nu \ge 0$ is the volatility, $r$ is the risk-free interest
rate, $\lambda \ge 0$ is the arrival intensity of a Poisson process, $\phi$ is 
the normal distribution with mean $\mu$ and standard deviation $\sigma$, and 
$\kappa = e^{\mu+\sigma^2/2} - 1$. 

The discretization of~\eqref{eq:merton} first truncates the infinite domain $(-\infty, \infty) \times [0, T]$
to $(-\xi_{\min}, \xi_{\max}) \times [0, T]$. Then central differences
and the rectangle method are used to discretize the differential and
integral terms in~\eqref{eq:merton}, respectively. Because the
coefficients do not depend on $\xi$ and the integral kernel is shift
invariant, the discretization yields a nonsymmetric, sectorial Toeplitz matrix
$T$. We refrain from giving details and refer to the excellent
summary given in~\cite[Example~3]{Lee2010}. 

In all experiments, we used parameters identical to the ones used in~\cite{Lee2010}: $\xi_{\min} = -2$, $\xi_{\max} = 2$, $K = 100$, $\nu = 0.25$, $r = 0.05$,
$\lambda = 0.1$, $\mu = −0.9$, $\sigma = 0.45$, as well as a full time step $1$. 

\begin{figure}[t]
\includegraphics[width=0.49\textwidth]{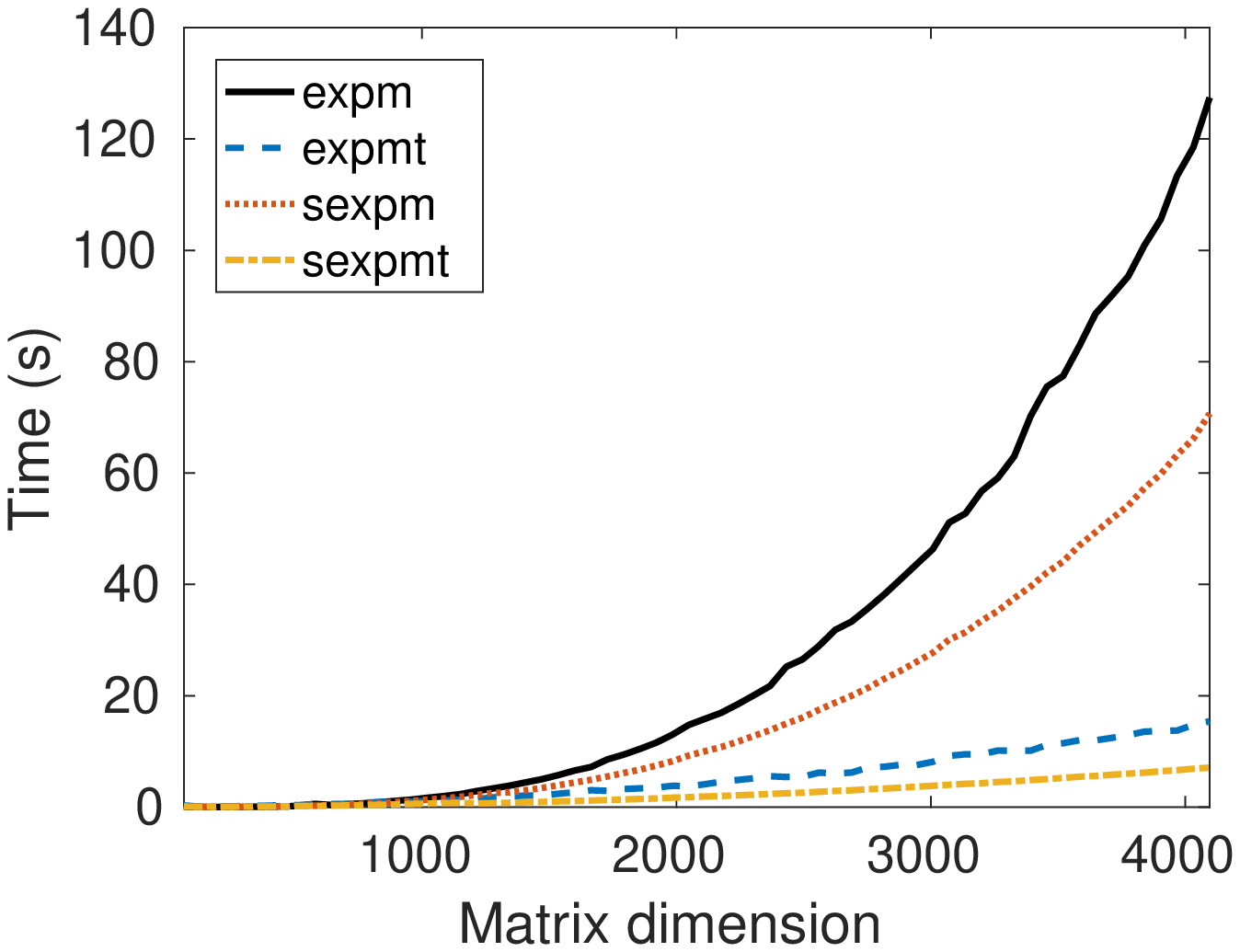} \hfill
\includegraphics[width=0.49\textwidth]{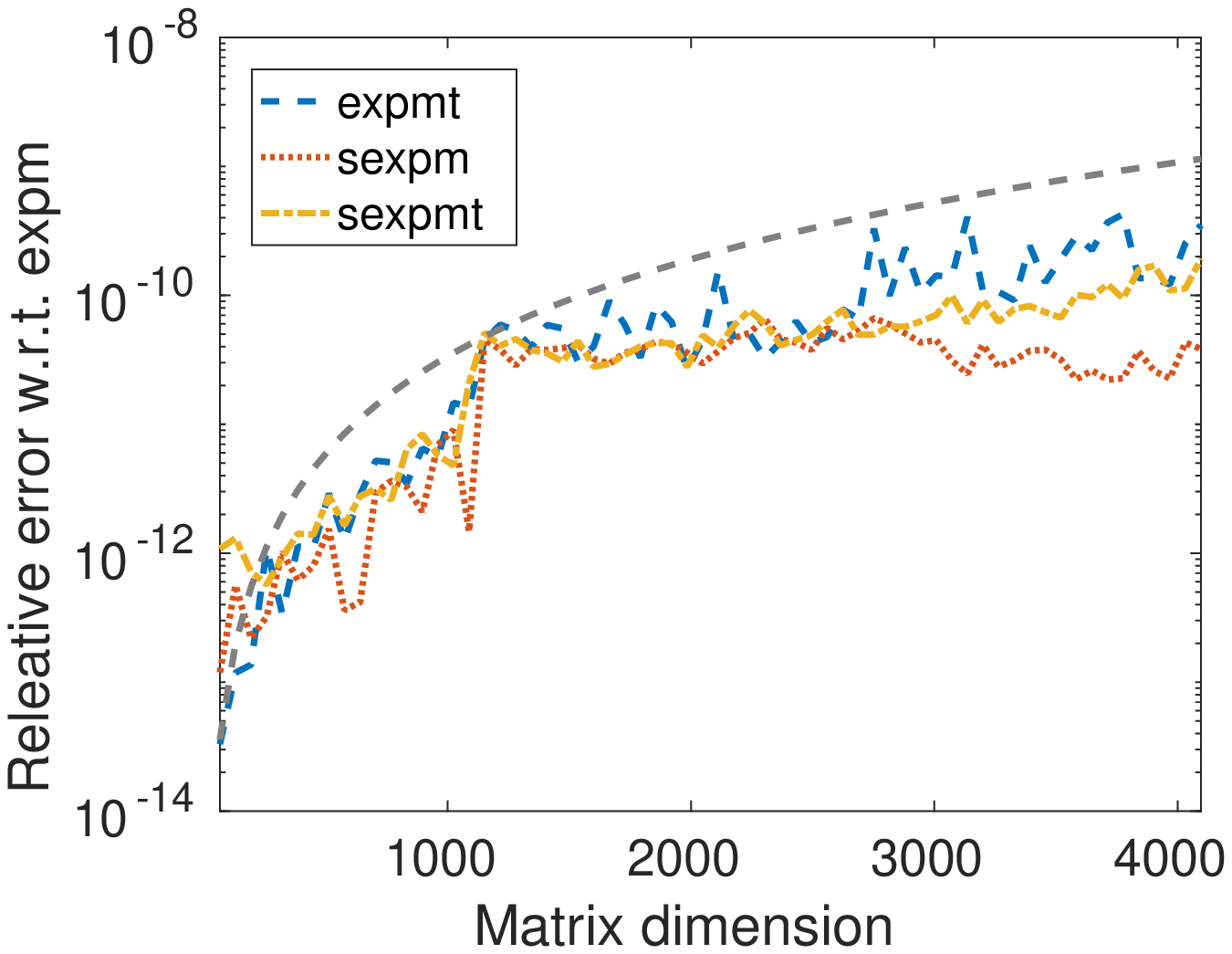}
\caption{Experimental comparison of approximations to $\exp(A_n)$ for
the Merton model (see Section~\ref{sec:merton}).  \emph{Left:} Run
    time comparison
for the computation of the full exponential approximation.
\emph{Right:} Relative error with respect to {\ttfamily expm}.  The
dashed gray line shows $u \norm{A_n}_F$, a lower bound for the
exponential condition number (times machine precision).
\label{fig:sexpmt_merton}}
\end{figure}

Figure~\ref{fig:sexpmt_merton} (left) shows the run time for the four
matrix exponential approximations {\ttfamily expm}, {\ttfamily expmt},
{\ttfamily sexpm}, and {\ttfamily sexpmt}.  From the complexity
analysis in Section~\ref{sec:algtools} we know that the run times of
{\ttfamily expmt} and {\ttfamily sexpmt} scale quadratically in $n$
({\ttfamily expm} and {\ttfamily sexpm} scale cubically), and the plot
shows this qualitative different behaviour.  The run time scaling
exponents inferred from the measured times $t_1, t_2$ at dimensions
$n_1 = 2048$ and $n_2=4096$ through $\log(t_2/t_1) / \log(n_2 / n_1)$,
is about $2.20$ for {\ttfamily expmt} and $2.00$ for {\ttfamily
sexpmt}.  From the shown data we also see that a saving in run time by
using our fast algorithm over the standard ones is realized already
for matrix sizes between $n=1000$ and $n=1500$, indicating that the
constants hidden in the big-$\OOM$ complexity bounds for our
algorithms are not too large.

We now discuss the accuracy of the obtained approximations.  Since the
computation of the matrix exponential using variable precision
arithmetic is too expensive for the matrix sizes we are considering
here, we assess the accuracy of {\ttfamily expmt}, {\ttfamily sexpm},
and {\ttfamily sexpmt} with reference to {\ttfamily expm}.
Figure~\ref{fig:sexpmt_merton} (right), shows the relative distance
\begin{equation*} \frac{\norm{B - \expm(T)}_F}{\norm{\expm(T)}_F},
\end{equation*} where $B$ is the approximation we want to compare.  In
addition, we show $u\norm{T}_F$ (dashed line), which is a lower bound
for the exponential condition number (times machine precision).  Since
the relative error w.r.t.~{\ttfamily expm} is roughly bounded by this
quantity, we conclude that our adapted scaling and squaring methods
behave in a forward stable manner in this example.

\begin{figure} \begin{center}
\includegraphics[width=0.49\textwidth]{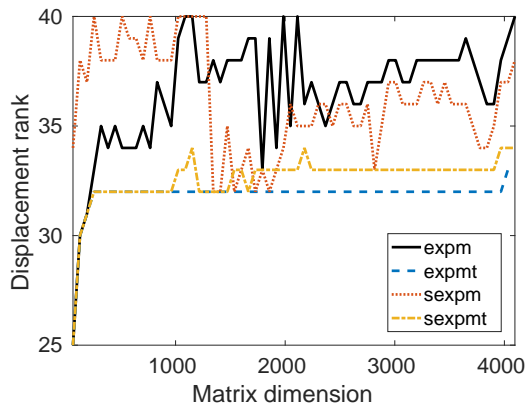}
\end{center} \caption{Displacement ranks of the exponentials computed
for the Merton model.\label{fig:merton_ranks}} \end{figure}

\begin{figure} \begin{center}
\includegraphics[width=0.49\textwidth]{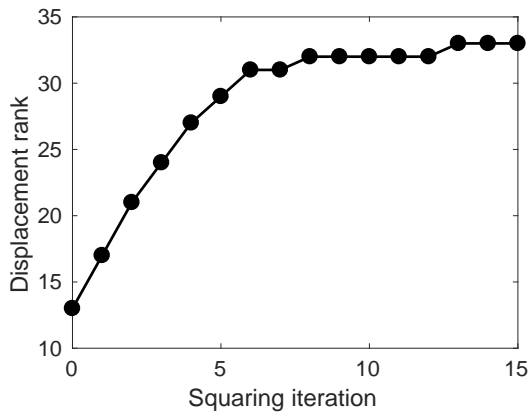}
\end{center} \caption{Evolution of the displacement ranks in the
squaring phase of Alg.~\ref{alg:expmt} for the Merton model
discretized with $n=4096$ points.  \label{fig:squaring_ranks}}
\end{figure}

In Figure~\ref{fig:merton_ranks} we show the displacement ranks of the
approximations to the matrix exponentials.  For {\ttfamily expmt} and
{\ttfamily sexpmt}, this rank corresponds to the length of the
generator obtained by the corresponding method after the squaring
phase.  For {\ttfamily expm} and {\ttfamily sexpm}, the shown rank is
the numerical rank of the displacements, as determined by
\textsc{Matlab}'s {\ttfamily rank} function.  As suggested by the
discussion in Section~\ref{sec:rationalapproximation}, all these ranks
are close to each other, and in particular quite small.  Finally,
Figure~\ref{fig:squaring_ranks} shows how the displacement rank
evolves during the squaring phase of {\ttfamily expmt} ($n =
4096$).

\section{Conclusions and future work}

We have shown that the full matrix exponential of a Toeplitz matrix can be
computed efficiently using scaling and squaring algorithms.  A key result that
enables this efficient computation is the low displacement rank of rational
functions for Toeplitz matrices.  Combined with classical results for rational
best approximations of the exponential function, it asserts that the Toeplitz
matrix exponential itself enjoys provable low displacement rank if its spectrum
is real or sectorial, for example.  By carefully adapting all the matrix
computations in the general scaling and squaring framework, we obtain
algorithms of quadratic complexity of the input size for computing $\exp(T)$
for a Toeplitz matrix $T$.  Since the output size of the matrix exponential is
quadratic as well, our algorithms hence achieve \emph{optimal} complexity.

We also demonstrated by means of an example that a spectrum clustered along a
long stretch of the imaginary axis does not allow for a low displacement rank
approximation of the matrix exponential.  The apparent reason is the
periodicity of the exponential function on this set, which does not allow for a
low degree rational approximation.  It would be of interest to investigate this
setting more rigorously.

In this work we have focused on analyzing the displacement rank of
polynomials, rational functions and the matrix exponential itself.
Two important aspects have received less attention than they probably
deserve.  One is the design of the scaling and squaring logic itself, for
which we relied on the works of Higham~\cite{Higham2009}, as well as
G{\"u}ttel and Nakatsukasa~\cite{Guttel2016}.  An important
design goal for these methods is a small number of (unstructured)
matrix operations of cubic complexity such as inversion or
matrix-matrix multiplication.  However, as our careful description in
Section~\ref{sec:algtools} shows, minimizing these operations is of
much less importance in the Toeplitz case, as long as the overall
quadratic complexity is maintained.  It would thus be of interest to
design a scaling and squaring method specifically for the case of
Toeplitz matrices.  We also did not attempt to analyze the forward
stability in floating point arithmetic for our adapted methods in
detail, although such an analysis is certainly of interest.

Finally, our results show that for Toeplitz matrices it is even
possible to implement scaling and squaring algorithms of
\emph{subquadratic complexity}, provided that only the generators of
$\exp(T)$ are requested, and not the full matrix exponential.  Recall
that the generators are already sufficient for applying the
exponential to a vector.  If, for example, the Toepliz inversions in
Alg.~\ref{alg:sexpmt} are carried out by superfast solvers
(e.g.,~\cite{Kailath1999,Olshevsky2001a,Xia2012}), then these
generators can be computed in $\OOM(n \log n)$.

\bibliographystyle{siamplain}
\bibliography{expmt}

\begin{thebibliography}{10}

\bibitem{Arico2010}
{\sc A.~Aric{\`o} and G.~Rodriguez}, {\em A fast solver for linear systems with
  displacement structure}, Numer. Algorithms, 55 (2010), pp.~529--556,
  \url{http://dx.doi.org/10.1007/s11075-010-9421-x}.

\bibitem{Bai2000}
{\sc Z.~Bai, J.~Demmel, J.~Dongarra, A.~Ruhe, and H.~van~der Vorst}, eds., {\em
  Templates for the solution of algebraic eigenvalue problems}, vol.~11 of
  Software, Environments, and Tools, Society for Industrial and Applied
  Mathematics (SIAM), Philadelphia, PA, 2000,
  \url{http://dx.doi.org/10.1137/1.9780898719581}.

\bibitem{Bini2015}
{\sc D.~A. Bini, S.~Dendievel, G.~Latouche, and B.~Meini}, {\em Computing the
  exponential of large block-triangular block-{T}oeplitz matrices encountered
  in fluid queues}, Linear Algebra Appl., 502 (2016), pp.~387--419,
  \url{http://dx.doi.org/10.1016/j.laa.2015.03.035}.

\bibitem{Chandrasekaran1998}
{\sc S.~Chandrasekaran and A.~H. Sayed}, {\em A fast stable solver for
  nonsymmetric {T}oeplitz and quasi-{T}oeplitz systems of linear equations},
  SIAM J. Matrix Anal. Appl., 19 (1998), pp.~107--139,
  \url{http://dx.doi.org/10.1137/S0895479895296458}.

\bibitem{Duffy2006}
{\sc D.~J. Duffy}, {\em Finite difference methods in financial engineering},
  Wiley Finance Series, John Wiley \& Sons, Ltd., Chichester, 2006,
  \url{http://dx.doi.org/10.1002/9781118673447}.
\newblock A partial differential equation approach, With 1 CD-ROM (Windows,
  Macintosh and UNIX).

\bibitem{Gohberg1995}
{\sc I.~Gohberg, T.~Kailath, and V.~Olshevsky}, {\em Fast {G}aussian
  elimination with partial pivoting for matrices with displacement structure},
  Math. Comp., 64 (1995), pp.~1557--1576,
  \url{http://dx.doi.org/10.2307/2153371}.

\bibitem{Gonchar1987}
{\sc A.~A. Gonchar and E.~A. Rakhmanov}, {\em Equilibrium distributions and the
  rate of rational approximation of analytic functions}, Mat. Sb. (N.S.),
  134(176) (1987), pp.~306--352, 447.

\bibitem{Gu1998}
{\sc M.~Gu}, {\em Stable and efficient algorithms for structured systems of
  linear equations}, SIAM J. Matrix Anal. Appl., 19 (1998), pp.~279--306
  (electronic), \url{http://dx.doi.org/10.1137/S0895479895291273}.

\bibitem{Guttel2013a}
{\sc S.~G{\"u}ttel}, {\em Rational {K}rylov approximation of matrix functions:
  numerical methods and optimal pole selection}, GAMM-Mitt., 36 (2013),
  pp.~8--31, \url{http://dx.doi.org/10.1002/gamm.201310002}.

\bibitem{Guttel2016}
{\sc S.~G{\"u}ttel and Y.~Nakatsukasa}, {\em Scaled and squared subdiagonal
  {P}ad\'e approximation for the matrix exponential}, SIAM J. Matrix Anal.
  Appl., 37 (2016), pp.~145--170, \url{http://dx.doi.org/10.1137/15M1027553}.

\bibitem{Heinig1995}
{\sc G.~Heinig}, {\em Inversion of generalized {C}auchy matrices and other
  classes of structured matrices}, in Linear algebra for signal processing
  ({M}inneapolis, {MN}, 1992), vol.~69 of IMA Vol. Math. Appl., Springer, New
  York, 1995, pp.~63--81, \url{http://dx.doi.org/10.1007/978-1-4612-4228-4_5}.

\bibitem{Heinig1984}
{\sc G.~Heinig and K.~Rost}, {\em Algebraic methods for {T}oeplitz-like
  matrices and operators}, vol.~13 of Operator Theory: Advances and
  Applications, Birkh\"auser Verlag, Basel, 1984,
  \url{http://dx.doi.org/10.1007/978-3-0348-6241-7}.

\bibitem{Higham2008}
{\sc N.~J. Higham}, {\em Functions of matrices}, Society for Industrial and
  Applied Mathematics (SIAM), Philadelphia, PA, 2008,
  \url{http://dx.doi.org/10.1137/1.9780898717778}.
\newblock Theory and computation.

\bibitem{Higham2009}
{\sc N.~J. Higham}, {\em The scaling and squaring method for the matrix
  exponential revisited}, SIAM Rev., 51 (2009), pp.~747--764,
  \url{http://dx.doi.org/10.1137/090768539}.

\bibitem{Hochbruck1997}
{\sc M.~Hochbruck and C.~Lubich}, {\em On {K}rylov subspace approximations to
  the matrix exponential operator}, SIAM J. Numer. Anal., 34 (1997),
  pp.~1911--1925, \url{http://dx.doi.org/10.1137/S0036142995280572}.

\bibitem{Huckle1998}
{\sc T.~Huckle}, {\em Computations with {G}ohberg-{S}emencul-type formulas for
  {T}oeplitz matrices}, Linear Algebra Appl., 273 (1998), pp.~169--198,
  \url{http://dx.doi.org/10.1016/S0024-3795(97)00372-8}.

\bibitem{Kailath1994}
{\sc T.~Kailath and J.~Chun}, {\em Generalized displacement structure for
  block-{T}oeplitz, {T}oeplitz-block, and {T}oeplitz-derived matrices}, SIAM J.
  Matrix Anal. Appl., 15 (1994), pp.~114--128,
  \url{http://dx.doi.org/10.1137/S0895479889169042}.

\bibitem{Kailath1995}
{\sc T.~Kailath and A.~H. Sayed}, {\em Displacement structure: theory and
  applications}, SIAM Rev., 37 (1995), pp.~297--386,
  \url{http://dx.doi.org/10.1137/1037082}.

\bibitem{Kailath1999}
{\sc T.~Kailath and A.~H. Sayed}, eds., {\em Fast reliable algorithms for
  matrices with structure}, Society for Industrial and Applied Mathematics
  (SIAM), Philadelphia, PA, 1999,
  \url{http://dx.doi.org/10.1137/1.9781611971354}.

\bibitem{Lee2010}
{\sc S.~T. Lee, H.-K. Pang, and H.-W. Sun}, {\em Shift-invert {A}rnoldi
  approximation to the {T}oeplitz matrix exponential}, SIAM J. Sci. Comput., 32
  (2010), pp.~774--792, \url{http://dx.doi.org/10.1137/090758064}.

\bibitem{LopezFernandez2006}
{\sc M.~L{\'o}pez-Fern{\'a}ndez, C.~Palencia, and A.~Sch{\"a}dle}, {\em A
  spectral order method for inverting sectorial {L}aplace transforms}, SIAM J.
  Numer. Anal., 44 (2006), pp.~1332--1350 (electronic),
  \url{http://dx.doi.org/10.1137/050629653}.

\bibitem{Merton1976}
{\sc R.~C. Merton}, {\em Option pricing when underlying stock returns are
  discontinuous}, Journal of Financial Economics, 3 (1976), pp.~125--144,
  \url{http://dx.doi.org/10.1016/0304-405X(76)90022-2}.

\bibitem{Ng2000}
{\sc M.~K. Ng}, {\em Preconditioned {L}anczos methods for the minimum
  eigenvalue of a symmetric positive definite {T}oeplitz matrix}, SIAM J. Sci.
  Comput., 21 (2000), pp.~1973--1986 (electronic),
  \url{http://dx.doi.org/10.1137/S1064827597330169}.

\bibitem{Olshevsky2001a}
{\sc V.~Olshevsky}, ed., {\em Structured matrices in mathematics, computer
  science, and engineering. {I}}, vol.~280 of Contemporary Mathematics,
  American Mathematical Society, Providence, RI, 2001.

\bibitem{Olshevsky2001b}
{\sc V.~Olshevsky}, ed., {\em Structured matrices in mathematics, computer
  science, and engineering. {II}}, vol.~281 of Contemporary Mathematics,
  American Mathematical Society, Providence, RI, 2001.

\bibitem{Pan1993}
{\sc V.~Pan}, {\em Decreasing the displacement rank of a matrix}, SIAM J.
  Matrix Anal. Appl., 14 (1993), pp.~118--121,
  \url{http://dx.doi.org/10.1137/0614010}.

\bibitem{Redivo2012}
{\sc M.~Redivo-Zaglia and G.~Rodriguez}, {\em {smt}: {M}atlab toolbox for
  structured matrices}, Numer. Algorithms, 59 (2012), pp.~639--659,
  \url{http://dx.doi.org/10.1007/s11075-011-9527-9}.

\bibitem{Sachs2008}
{\sc E.~W. Sachs and A.~K. Strauss}, {\em Efficient solution of a partial
  integro-differential equation in finance}, Appl. Numer. Math., 58 (2008),
  pp.~1687--1703, \url{http://dx.doi.org/10.1016/j.apnum.2007.11.002}.

\bibitem{Sayed1995}
{\sc A.~H. Sayed and T.~Kailath}, {\em A look-ahead block {S}chur algorithm for
  {T}oeplitz-like matrices}, SIAM J. Matrix Anal. Appl., 16 (1995),
  pp.~388--414, \url{http://dx.doi.org/10.1137/S0895479892232649}.

\bibitem{Trefethen2014}
{\sc L.~N. Trefethen and J.~A.~C. Weideman}, {\em The exponentially convergent
  trapezoidal rule}, SIAM Rev., 56 (2014), pp.~385--458,
  \url{http://dx.doi.org/10.1137/130932132}.

\bibitem{Xia2012}
{\sc J.~Xia, Y.~Xi, and M.~Gu}, {\em A superfast structured solver for
  {T}oeplitz linear systems via randomized sampling}, SIAM J. Matrix Anal.
  Appl., 33 (2012), pp.~837--858, \url{http://dx.doi.org/10.1137/110831982}.

\end{thebibliography}

\end{document}